\documentclass[preprint,5p,times,twocolumn]{elsarticle}
\usepackage{amssymb}
\usepackage{amsthm}
\usepackage{amsfonts}
\usepackage{graphicx}
\usepackage{stmaryrd}
\usepackage{bigints}
\usepackage{booktabs}
\usepackage[makeroom]{cancel}
\usepackage{url}
\usepackage[open]{bookmark}
\allowdisplaybreaks
\newtheorem{thm}{Theorem}
\newtheorem{lem}{Lemma}
\newtheorem{prop}{Proposition}
\newtheorem{crly}{Corollary}
\newtheorem{assm}{Assumption}
\newtheorem{rmk}{Remark}
\usepackage{color,soul}
\usepackage{tikz}
\DeclareMathOperator{\Hinf}{\mathcal{H}_{\infty}}
\DeclareMathOperator{\G}{\mathcal{G}}
\newcommand{\RNum}[1]{\uppercase\expandafter{\romannumeral #1\relax}}
\usepackage{optidef}
\usepackage{subfig}
\definecolor{bluee}{rgb}{0.00,0.45,0.74}
\usepackage{fullpage}
\usepackage{flushend}

\pdfoutput=1

\usepackage{etoolbox}
\patchcmd{\MaketitleBox}{\footnotesize\itshape\elsaddress\par\vskip36pt}{\footnotesize\itshape\elsaddress\par\parbox[b][36pt]{\linewidth}{\vfill\hfill\textnormal{
\fbox{
Accepted for publication: European Journal of Control. (DOI: 10.1016/j.ejcon.2021.06.032) \textcopyright 2021 Elsevier.
}
}\hfill\null\vfill}}{}{}%

\journal{European Journal of Control}

\begin{document}

\begin{frontmatter}

\title{$\mathcal{H}_\infty$ Network Optimization for Edge Consensus\tnoteref{footnote}}

\tnotetext[footnote]{A preliminary version of this paper was presented at the 2021 European Control Conference. O. Farhat is a recipient of the 2020 ESED scholarship by the Global Sustainable Electricity Partnership. D. Abou Jaoude acknowledges the support of the University Research Board (URB) at the American University of Beirut (AUB). This research is supported by the SNSF through NCCR Automation and the ETH  Foundation.}

\author[AUB]{Omar Farhat}\ead{off03@mail.aub.edu}    
\author[AUB]{Dany Abou Jaoude\corref{corres}}\ead{da107@aub.edu.lb}               
\cortext[corres]{Corresponding author.}
\author[ETH]{Mathias Hudoba de Badyn}\ead{mbadyn@ethz.ch}  

\address[AUB]{Control and Optimization Lab, Department of Mechanical Engineering, Maroun Semaan Faculty of Engineering and Architecture, American University of Beirut, Lebanon}  
\address[ETH]{Automatic Control Laboratory, Swiss Federal Institute of Technology, Z\"{u}rich, Physikstrasse 3, 8092, Switzerland}

\begin{abstract}

This paper examines the $\Hinf$ performance problem of the edge agreement protocol for networks of agents operating on independent time scales, connected by weighted edges, and corrupted by exogenous disturbances. 
$\Hinf$-norm expressions and bounds are computed that are then used to derive new insights on network performance in terms of the effect of time scales and edge weights on disturbance rejection. 
We use our bounds to formulate a convex optimization problem for time scale and edge weight selection. Numerical examples are given to illustrate the applicability of the derived $\Hinf$-norm bound expressions, and the optimization paradigm is illustrated via a formation control example involving non-homogeneous agents. 
\end{abstract}


\end{frontmatter}

\section{INTRODUCTION}
Many natural and engineered systems operate over protocols or dynamics over networks.
The consensus algorithm is a famous distributed information-sharing protocol over a network, used in many applications ranging from wind farm optimization~\cite{annoni2018framework}, robotics and autonomous spacecraft~\cite{Rossi2018}, sensor networks and compressed sensing~\cite{Olfati-Saber2005,Schmidt2009}, and multi-agent systems~\cite{Behavior2007,Tanner2004}.
There has been extensive research on how the structure of the network affects control-theoretic properties of the collective system, such as optimal edge weight selection for $\mathcal{H}_2$ control~\cite{Chapman2015,Foight2019ts}, as well as the relationship between network symmetries and controllability~\cite{Rahmani2009a,Alemzadeh2017}.

The present paper is focused on system-theoretic robustness measures analyzing the consensus algorithm, and in particular how one may design robust consensus networks.
The state-of-the-art in the literature has focused on the $\mathcal{H}_2$ performance measure, which describes how much energy enters a system via an impulse response, or equivalently, how well a system rejects noise.
In leader-follower consensus, the $\mathcal{H}_2$-norm captures the notion of effective resistance across the network~\cite{Barooah2008,Chapman2013a}.
This paradigm proves useful in analyzing the effects of optimizing edge weights, node time scales, and the graph structure for $\mathcal{H}_2$ performance~\cite{Chapman2015,Foight2019ts,Zelazo2011a,HudobadeBadyn2021,DeBadyn2020}.
The related concept of coherence has been used to develop local feedback laws and for leader selection~\cite{Bamieh2012,Patterson2010a,Patterson2014}.

While the $\mathcal{H}_2$-norm captures noise-rejection properties of the system, in the present work we examine the $\mathcal{H}_\infty$-norm, which is related to how finite-energy signals and disturbances cause the system to deviate from consensus~\cite{Zelazo2011a}.
The $\mathcal{H}_\infty$-norm has also found use in combating time delays in consensus problems~\cite{Lin2008}, as well as in distributed state estimation over sensor networks~\cite{Shen2010}.

When only relative measurements across edges in the network are available to the agents, a minimal representation of consensus can be considered, which we refer to as edge consensus \cite{Zelazo2011a}. This minimal representation is characterized by the edge Laplacian of the graph, and is used in formation flight and sensor networks~\cite{Sandhu2005,Sandhu2009,Smith2005}. The edge consensus protocol has been studied in \cite{Foight2019a}, and extensions of edge consensus for matrix-weighted networks were considered in~\cite{Foight2019ts}.

Networks often contain non-homogeneous agents that operate on different time scales.
Examples of agent-based dynamics operating on multiple time scales include coupled oscillators~\cite{nawrath2010distinguishing}, neuronal networks~\cite{honey2007network}, and social networks~\cite{flack2012multiple}.
Theoretical tools exploiting time scale analysis have also been developed for various systems, such as neural network-based control for robotics~\cite{yamashita2008emergence}, model reduction and coherency in power systems~\cite{romeres2013novel,chow1990singular}, epidemic spreading~\cite{lewien2019time}, and composite (or layered) consensus networks~\cite{Chapman2016}. Interaction with time-scaled networks was discussed in \cite{Foight2019}, and consensus under time scale separation was presented in \cite{Awad2018}.

A natural question to ask is what distribution of time scales in a network yields resiliency, either to disturbances, noise, or network failure. A common simplifying approach to address such theoretical questions is to assume that the agents in the network are single integrator units \cite{Mesbahi2010}. This approach is followed in the present work. Our paper also assumes the same choice of covariance matrices as in \cite{Foight2019ts,Foight2019a}. Based on this assumption, one contribution in our paper is to quantify the effect of time scales in the edge consensus protocol on disturbance rejection, using the $\mathcal{H}_\infty$-norm as the performance measure. Related works in the literature considering edge consensus for time-scaled networks are~\cite{Foight2019ts,HudobadeBadyn2021,Foight2019a}.

Our present work focuses on examining the $\mathcal{H}_\infty$-norm in the case of edge consensus on both edge-weighted and time-scaled networks.
The contributions of our paper are:
\begin{enumerate}
    \item We derive expressions of and bounds on the $\mathcal{H}_\infty$-norm in terms of edge weights and time scales.
    \item We make use of the derived expressions and bounds to provide new insights on $\Hinf$ network performance; namely, that larger edge weights and faster time scales lead to a smaller $\mathcal{H}_\infty$-norm.
    \item We show how the $\Hinf$-norm bounds provide a suitable proxy for formulating optimization problems for the selection of edge weights and time scales.
    \item We provide numerical examples to illustrate:
      \begin{enumerate}
        \item The validity and conservativeness of the derived bound expressions
        \item A Pareto optimal front as a function of the three tuning parameters in our edge weight and time scale optimization framework
        \item The performance improvement under our edge weight and time scale optimization framework in a formation control example.
      \end{enumerate}
\end{enumerate}

The organization of this paper is as follows.
In \S\ref{sec:preliminaries}, we outline our notation and the mathematical preliminaries on graph-theoretic properties. The edge consensus setup is described in \S\ref{sec:problemsetup}, and our main results are given in \S\ref{sec:main-results}. A formation control consensus
example is given in \S\ref{sec:example}, and the paper is concluded in \S\ref{sec:conclusion}.

\section{NOTATION AND PRELIMINARIES}
\label{sec:preliminaries}
In this section, we outline the notation used in the paper. The set of real numbers is denoted by $\mathbb{R}$. The sets of $n$-dimensional real vectors and  $m\times n$ real matrices are denoted by $\mathbb{R}^{n}$ and $\mathbb{R}^{m\times n}$, respectively. Diagonal matrices are written as $D=\mathbf{diag}(d)=\mathbf{diag}\{d_{1},\ldots,d_{n}\}$, where $d=(d_{1},\ldots,d_{n})\in\mathbb{R}^{n}$. $\mathbf{1}$ and $\mathbf{0}$ denote vectors that consist of all one and zero entries, respectively. The identity matrix of conformable dimensions is denoted by $I$. $|S|$ denotes the cardinality of a set $S$. We define the $\ell_{2}$-norm of a vector $x\in\mathbb{R}^{n}$ as $\|x\|_{2}=(x_{1}^{2}+\cdots+x_{n}^{2})^{\frac{1}{2}}$. $\mathbf{det}(A)$ denotes the determinant of matrix $A$. $C\succ 0$ and $C\succeq 0$ mean that the symmetric matrix $C$ is positive definite and positive semidefinite, respectively.

The null space of a matrix $A$ is denoted by $\mathcal{N}(A)$. $\mathbf{span}(A)$ denotes the subspace spanned by the columns of $A$. $M^T$, $M^{-1}$, and $M^{\dagger}$ denote the transpose, inverse, and Moore-Penrose pseudoinverse of a matrix $M$, respectively. For a nonsingular matrix $A$, $(A^{T})^{-1}$ is simplified to $A^{-T}$. Two matrices $C$ and $D$ are similar if there exists a nonsingular matrix $T$ such that $D=T^{-1}CT$. In this case, $C$ and $D$ have the same spectra. The maximum and minimum eigenvalues of a symmetric matrix $A$ are denoted by $\lambda_{\rm{max}}(A)$ and $\lambda_{\rm{min}}(A)$, respectively. The largest singular value of a (possibly nonsquare) complex matrix $B$ is defined as $\bar{\sigma}(B)=\sqrt{\lambda_{\rm{max}}(B^{*}B)}$, where $B^{*}$ denotes the conjugate transpose of $B$. For $A\succeq 0$, $\bar{\sigma}(A)=\lambda_{\rm{max}}(A)$, and so, these two terms will be used interchangeably in this case. 

This paper considers undirected, connected, and weighted graphs that are comprised of nodes with multiple time scales and no self-loops. 
The quadruple $\G=(\mathcal{V},\mathcal{E},W,E)$ denotes a graph $\G$, where $\mathcal{V}=\{1,\ldots,n\}$ is the set of nodes, $\mathcal{E}$ is the set of edges, $W$ is a diagonal matrix of (positive) edge weights, and $E$ is a diagonal matrix of (positive) node time scales. If we partition the set of nodes $\mathcal{V}$ into two disjoint sets $\mathcal{V}_1$ and $\mathcal{V}_2$, we write that $\mathcal{V} = \mathcal{V}_1 \sqcup \mathcal{V}_2$, i.e.,~ $\mathcal{V}$ is the disjoint union of $\mathcal{V}_1$ and $\mathcal{V}_2$.
We uniquely label each edge $ij$ as $l\in\{1,\ldots,|\mathcal{E}|\}$ and denote the associated weight as $w_{ij}=w_{ji}=w_{l}>0$. 
The weight matrix is thus defined as $W=\mathbf{diag}\{w_{1},\ldots,w_{|\mathcal{E}|}\}$. 
The time scale associated with a node $i$ is denoted by $\epsilon_{i}>0$, and so the matrix of time scales is defined as $E=\mathbf{diag}\{\epsilon_{1},\ldots,\epsilon_{|\mathcal{V}|}\}$. Let $N(i)$ denote the set of neighbor nodes of node $i$, i.e., the nodes $j$ such that $ij\in\mathcal{E}$.
The incidence matrix $D(\G)$ is a $|\mathcal{V}|\times|\mathcal{E}|$ matrix that characterizes the incidence relation between distinct pairs of nodes. 
For undirected graphs, the incidence matrix is constructed by giving the graph an arbitrary orientation, and therefore, the edges will have terminal and initial nodes. Then, $D(\G)$ can be defined as: $\begin{bmatrix}D(\G)\end{bmatrix}_{il}=1$ if $i$ is the initial node of edge $l$, $-1$ if $i$ is the terminal node of edge $l$, and $0$ otherwise.

A connected graph $\G$ can be split into two edge-disjoint subgraphs $\G_{\tau}$ and $\G_{c}$ on the same set of nodes, where $\G_{\tau}$ is a spanning tree (a connected graph on the $n$ nodes with $n-1$ edges) and $\G_{c}$ is the corresponding co-tree. 
$\G_{\tau}$ contains all the nodes of $\G$ connected with $|\mathcal{V}|-1$ edges, and $\G_{c}$ contains the remaining edges of $\G$. A connected graph has at least one spanning tree \cite{GTSC}. As a result, the columns of the incidence matrix can be permuted such that $D(\G)=\begin{bmatrix}D(\G_{\tau}) & D(\G_{c})\end{bmatrix}$ without loss of generality. Co-tree edges are linear combinations of the tree edges, i.e., $D(\G_{\tau})T_{\tau}^{c}=D(\G_{c})$, where the matrix $T_{\tau}^{c}=(D(\G_{\tau})^{T}D(\G_{\tau}))^{-1}D(\G_{\tau})^{T}D(\G_{c})$. Thus, the incidence matrix can be expressed as $D(\G)=D(\G_{\tau})R(\G)$, where $R(\G)=\begin{bmatrix}I & T_{\tau}^{c}\end{bmatrix}$, as in \cite{GTSC}. Other graph-associated matrices are the graph Laplacian $L(\G)=D(\G)D(\G)^{T}$ and the edge Laplacian $L_{e}(\G)=D(\G)^{T}D(\G)$.

\section{EDGE CONSENSUS}
\label{sec:problemsetup}
In this section, we outline the edge variant of the consensus protocol over a network of time-scaled agents interconnected by weighted edges. We also introduce exogenous inputs in the form of measurement and process noise. The reader is referred to \cite{Foight2019a} for further details.

Consider a group of \emph{single integrator} units evolving at differing rates. We represent this configuration by a graph $\G=(\mathcal{V},\mathcal{E},W,E)$, where agents and the interconnections between them are represented by nodes and edges, respectively. Let $t\geq 0$ denote continuous time. The dynamics of agent $i\in\mathcal{V}$ are given by
\begin{equation*}
\label{eq:stateeqn}
\epsilon_{i}\dot{x}_{i}(t)=u_{i}(t)+w_{i}(t),
\end{equation*}
where $x_{i}(t)$, $\epsilon_{i}$, $u_{i}(t)$, and $w_{i}(t)$ represent the state, time scale, input, and zero-mean Gaussian process noise associated with agent $i$, respectively. The total input $u_i(t)$ to agent $i$ is the sum of the control input that achieves consensus and the corrupting measurement noise on the edges. For $i\in\mathcal{V}$, the input $u_{i}(t)$ is thus defined as
\begin{equation*}
\label{eq:controlinput}
u_{i}(t)=\sum_{j\in N(i)}(w_{ij}(x_{j}(t)-x_{i}(t))+v_{ij}(t)),
\end{equation*}
where $v_{ij}(t)$ is a zero-mean Gaussian measurement noise affecting edge $ij$.
In matrix form, the dynamics of all the agents and the stacked input are respectively expressed as
\begin{align}
\dot{x}(t)&=E(\G)^{-1}u(t)+E(\G)^{-1}w(t), \label{eq:stateeqnmatrix}\\
u(t)&=-D(\G)W(\G)D(\G)^{T}x(t)-D(\G)v(t), \label{eq:controlinputmatrix}
\end{align}
where $x(t)$, $u(t)$, $w(t)$, and $v(t)$ are the stacked vectors of states, inputs, and process and measurement noises, respectively. 
We denote the covariance matrices of $w(t)$ and $v(t)$ by $\Omega$ and $\Gamma$, respectively. Combining \eqref{eq:stateeqnmatrix} and \eqref{eq:controlinputmatrix}, we obtain the following time-scaled and weighted consensus problem:
\begin{equation}
\label{eq:consensus}
\dot{x}(t)={-}L_{w,s}(\G)x(t)+\begin{bmatrix} E(\G)^{-1} & -E(\G)^{-1}D(\G)\end{bmatrix}\left[\!\!\!\begin{array}{c}
w(t)\\
v(t)
\end{array}
\!\!\!\right]{,}
\end{equation}
where $L_{w,s}(\G)=E(\G)^{-1}D(\G)W(\G)D(\G)^{T}$ is the scaled and weighted Laplacian matrix.

For a connected graph, the graph Laplacian matrix $L(\G)$ has one zero eigenvalue and the rest are positive \cite{GTSC}, i.e., the nullity of $L(\G)$ is equal to $1$. It can be shown that $\mathcal{N}(L_{w,s}(\G))=\mathcal{N}(L(\G))=\mathcal{N}(D(\G_{\tau})^{T})$. Thus, the nullity of $L_{w,s}(\G)$ is also equal to $1$. Moreover, $L_{w,s}(\G)$ has nonnegative eigenvalues since it is similar to the positive semidefinite matrix $E(\G)^{-\frac{1}{2}}D(\G)W(\G)D(\G)^{T}E(\G)^{-\frac{1}{2}}$. Hence, $L_{w,s}(\G)$ has one zero eigenvalue, and the rest are positive. Thus, the state matrix $-L_{w,s}(\G)$ in $\eqref{eq:consensus}$ is not Hurwitz, which precludes analysis involving the $\Hinf$-norm. Therefore, following \cite{Zelazo2011a}, we use an edge variant of (\ref{eq:consensus}) to perform the $\Hinf$ analysis.

\begin{lem}\cite[Theorem 1]{Foight2019a}
\label{lemma:similarity}
Consider a connected graph $\G$ with a given spanning tree $\G_{\tau}$. The scaled and weighted graph Laplacian $L_{w,s}(\G)$ is similar to
\begin{equation*}
\left[\begin{array}{cc}
L_{e,s}(\G_{\tau})R(\G)W(\G)R(\G)^{T} & \mathbf{0}\\
\mathbf{0} & 0
\end{array}\right],
\end{equation*}
where $L_{e,s}(\G_{\tau})=D(\G_{\tau})^{T}E(\G)^{-1}D(\G_{\tau})$ is the time-scaled edge Laplacian for $\G_{\tau}$.
\end{lem}
Lemma \ref{lemma:similarity} is proved in \cite{Foight2019a} by constructing the needed similarity transformation $S_{v}(\G)$ as follows:
\begin{equation*}
S_{v}(\G)=\begin{bmatrix}E(\G)^{-1}D(\G_{\tau})(D(\G_{\tau})^{T}E(\G)^{-1}D(\G_{\tau}))^{-1} & \mathbf{1}\end{bmatrix}.
\end{equation*}
The upper-left block of the resultant matrix, i.e., $L_{e,s}(\G_{\tau})R(\G)W(\G)R(\G)^{T}$, has positive eigenvalues. Moreover, it follows from \cite[Observation 7.1.8]{JohnsonHorn_MatrixAnalysis} that $L_{e,s}(\G_{\tau})=D(\G_{\tau})^{T}E(\G)^{-1}D(\G_{\tau})\succ 0$ and $R(\G)W(\G)R(\G)^{T}\succ 0$ since $W(\G)\succ 0$, $E(\G)^{-1}\succ 0$, and $R(\G)^T$ and $D(\G_{\tau})$ are full column rank. Hereafter, we omit the dependence of $R$, $W$, $E$, $D$, $L_{w,s}$, and $S_{v}$ on $\G$. We also define the simplified symbols $D_{\tau}=D(\G_{\tau})$ and $L_{e,s}^{\tau}=L_{e,s}(\G_{\tau})$.

Applying $x_{e}(t)=S_{v}^{-1}x(t)$ yields the following edge interpretation of the consensus dynamics:
\begin{multline}
\label{eq:edgeconsensus}
        \dot{x}_{e}(t)=\left[
\begin{array}{cc}
-L_{e,s}^{\tau}RWR^{T} & \mathbf{0}\\
\mathbf{0} & 0
\end{array}
\right]x_{e}(t)\\+\left[\begin{array}{cc}
D_{\tau}^{T}E^{-1} & -L_{e,s}^{\tau}R\\
\frac{1}{\epsilon_{s}}\mathbf{1}^{T} & \mathbf{0}
\end{array}\right]\left[\begin{array}{c}
w(t)\\
v(t)
\end{array}
\right],
\end{multline}
where $\epsilon_{s}=\sum_{i=1}^{n}\epsilon_{i}$. The vector of \emph{edge states} in \eqref{eq:edgeconsensus} can be partitioned as $x_{e}(t)=(x_{\tau}(t),x_{\mathbf{1}}(t))$, where $x_{\tau}(t)\in\mathbb{R}^{n-1}$ is the vector of edge states associated with the spanning tree $\mathcal{G}_\tau$ and $x_{\mathbf{1}}(t)\in\mathbb{R}$ is the edge state in the agreement/consensus space, $\mathbf{span(1)}$. Thus, for a given spanning tree $\G_\tau$, the \emph{edge consensus} model $\Sigma_{\tau}$ corresponding to the spanning tree edge states is given by
\begin{equation}
\label{eq:edgeagreement1}   
\begin{cases}
\dot{x}_{\tau}(t)&\!\!\!\!={-}L_{{e,s}}^{\tau}RWR^{T}x_{\tau}(t)+\!D_{\tau}^{T}E^{-1}\Omega\hat{w}(t)-\!L_{e,s}^{\tau}R\Gamma\hat{v}(t){,}\\
z(t)&\!\!\!\!=R^{T}x_{\tau}(t),
\end{cases}
\end{equation}
where $\hat{w}(t)$ and $\hat{v}(t)$ are normalized noise signals that satisfy $w(t)=\Omega \hat{w}(t)$ and $v(t)=\Gamma\hat{v}(t)$, respectively. In the edge consensus model $\Sigma_{\tau}$ in \eqref{eq:edgeagreement1}, the output $z(t)\in\mathbb{R}^{|\mathcal{E}|}$ is the monitored performance signal. Due to the inclusion of $R^T$ in the output equation, $z(t)$ consists of the spanning tree edge states, as well as the co-tree edge states. A variation of the model in (\ref{eq:edgeagreement1}) is also considered in \cite{Zelazo2011a,Foight2019a}, in which the output is chosen as the vector of the spanning tree edge states, i.e., $z(t)=x_{\tau}(t) \in \mathbb{R}^{n-1}$.

\section{MAIN RESULTS}
\label{sec:main-results}
In this section, we characterize the $\Hinf$-norm of the edge agreement protocol developed in \S\ref{sec:problemsetup} and develop a design problem for selecting edge weights and time scales.
\subsection[H-infinity Performance]{$\Hinf$ Performance}
\label{sec:Hinf}
Consider a stable linear time-invariant system with a transfer function matrix $\Phi(s)$, where $s$ denotes the Laplace variable. The $\Hinf$-norm of this system is defined as follows \cite[Chapter 4.3]{robustcontrol}:
\begin{equation*}
\|\Phi\|_{\infty}=\sup_{\omega\in\mathbb{R}}\{\bar{\sigma}\big(\Phi(j\omega)\big)\}.    
\end{equation*}
The $\Hinf$-norm is equal to the $\mathcal{L}_{2}$-induced norm of the system. In the context of networked systems, the $\Hinf$-norm is a measure of the asymptotic deviation of the states of the agents from consensus as a result of finite-energy exogenous disturbances \cite{Zelazo2011a}.

The transfer function matrix of system $\Sigma_{\tau}$ defined in \eqref{eq:edgeagreement1} is given by
\begin{equation}
\label{eq:tfsigma}
\Sigma_{\tau}(s)=R^{T}(sI+L_{e,s}^{\tau}RWR^{T})^{-1}\begin{bmatrix}D_{\tau}^{T}E^{-1}\Omega & -L_{e,s}^{\tau}R\Gamma\end{bmatrix}.
\end{equation}

\begin{prop}
\label{prop:statematrixdiag}
Consider the system $\Sigma_{\tau}$ defined in \eqref{eq:edgeagreement1}. The state matrix $-L_{e,s}^{\tau}RWR^{T}$ is diagonalizable.
\end{prop}
\begin{proof}
We begin by proving that $L_{w,s}$ is diagonalizable by showing that it is similar to a diagonalizable matrix. Namely, $E^{\frac{1}{2}}L_{w,s}E^{-\frac{1}{2}}=E^{-\frac{1}{2}}DWD^{T}E^{-\frac{1}{2}}$, and the latter matrix is diagonalizable since it is a real symmetric matrix \cite[Theorem 4.1.5]{JohnsonHorn_MatrixAnalysis}. 
From Lemma \ref{lemma:similarity}, the matrix \begin{equation*}
    \left[\begin{array}{cc}
L_{e,s}^{\tau}RWR^{T} & \mathbf{0}\\
\mathbf{0} & 0
\end{array}\right]
\end{equation*} is similar to $L_{w,s}$, and so it is diagonalizable. A block-diagonal matrix is diagonalizable if and only if its diagonal blocks are diagonalizable \cite[Lemma 1.3.10]{JohnsonHorn_MatrixAnalysis}. So, we can conclude that $-L_{e,s}^{\tau}RWR^{T}$ is diagonalizable.
\end{proof}
From Proposition \ref{prop:statematrixdiag}, it follows that there exists a transformation matrix $T$ such that $\Lambda=T^{-1}L_{e,s}^{\tau}RWR^{T}T$ is diagonal. Applying $\hat{x}_{\tau}(t)=T^{-1}x_{\tau}(t)$, we get
\begin{equation}
\label{eq:edgeagreementdiag}
\begin{cases}
\dot{\hat{x}}_{\tau}(t)&\!\!\!\!=-\Lambda\hat{x}_{\tau}(t)+T^{-1}\begin{bmatrix}D_{\tau}^{T}E^{-1}\Omega & -L_{e,s}^{\tau}R\Gamma\end{bmatrix}\left[\!\!\begin{array}{c}
\hat{w}(t)\\
\hat{v}(t)
\end{array}
\!\!\right]{,}\\
z(t)&\!\!\!\!=R^{T}T\hat{x}_{\tau}(t).
\end{cases}
\end{equation}
From the proof of Proposition \ref{prop:statematrixdiag}, it can be further seen that $\Lambda$ is a diagonal matrix that consists of the nonzero eigenvalues of $L_{w,s}$. 

Consider a variation of the system in \eqref{eq:edgeagreementdiag} in which the output is chosen as the state vector $\hat{x}_{\tau}(t)$. This system, denoted by $H$, has the following transfer function representation:
\begin{equation}
H(s)=(sI+\Lambda)^{-1}T^{-1}\begin{bmatrix}D_{\tau}^{T}E^{-1}\Omega & -L_{e,s}^{\tau}R\Gamma\end{bmatrix}.
\label{eq:tfmodified}
\end{equation}
From \eqref{eq:edgeagreementdiag} and \eqref{eq:tfmodified}, it follows that $\Sigma_{\tau}(s)=R^{T}TH(s)$.
\begin{lem}
\label{lemma:zelazo1}
The $\Hinf$-norm of the system $H$ defined in \eqref{eq:tfmodified} satisfies $\|H\|_{\infty}=\bar{\sigma}\big(H(0)\big)$.
\end{lem}
\begin{lem}
\label{lemma:zelazo2}
The $\Hinf$-norm of the system $\Sigma_\tau$ defined in \eqref{eq:edgeagreement1}, and having the similar realization defined in \eqref{eq:edgeagreementdiag}, satisfies $\|\Sigma_{\tau}\|_{\infty}=\bar{\sigma}\big(R^{T}TH(0)\big)$.
\end{lem}
The proofs of Lemmas \ref{lemma:zelazo1} and \ref{lemma:zelazo2} are similar to those of Propositions $3.5$ and $3.6$ in \cite{Zelazo2011a}, respectively, and are omitted for brevity. From Lemma \ref{lemma:zelazo1}, it follows that the supremum for the modified system $H$ occurs at $\omega=0$. From Lemma \ref{lemma:zelazo2}, the supremum for the system $\Sigma_{\tau}$ also occurs at $\omega=0$. 

The following lemma allows for deriving a property to be used in subsequent proofs. To state this lemma, we define the majorisation $\prec_{m}$ as follows. Let $a$ and $b$ be two nonnegative vectors in $\mathbb{R}^c$. Then, $\log a\prec_{m} \log b$ means that $\prod_{i=1}^{k}a_{[i]}\leq\prod_{i=1}^{k}b_{[i]}$ for all $1\leq k\leq c$ and $\prod_{i=1}^{c}a_{[i]}=\prod_{i=1}^{c}b_{[i]}$, where $(a_{[1]},\ldots,a_{[c]})$ is the nonincreasing rearrangement of $a \in \mathbb{R}^{c}$ \cite[Chapter 3]{bhatia_matrixanalysis}.

\begin{lem}\cite[Corollary \RNum{3}.4.6]{bhatia_matrixanalysis}
\label{lemma:eigproperty}
Let $M$ and $N$ be two $c\times c$ positive semidefinite matrices. Then, the vector of eigenvalues $\lambda(MN)=(\lambda_{\rm{max}}(MN),\ldots,\lambda_{\rm{min}}(MN))$ is nonnegative and satisfies
\begin{align*}
&\log\lambda(MN)\prec_{m}\log\lambda^{\downarrow}(M)+\log\lambda^{\downarrow}(N),\\
&\log\lambda(MN)\succ_{m}\log\lambda^{\downarrow}(M)+\log\lambda^{\uparrow}(N),
\end{align*}
where $\lambda^{\downarrow}(M)$ and $\lambda^{\uparrow}(M)$ are the vectors of the eigenvalues of $M$ arranged in a nonicreasing and nondecreasing order, respectively.
\end{lem}

From the properties of the majorisation, and for $k=1$,  it follows that
\begin{equation}
\label{eq:bounds}
\lambda_{\rm{max}}(M)\lambda_{\rm{min}}(N)\leq\lambda_{\rm{max}}(MN)\leq\lambda_{\rm{max}}(M)\lambda_{\rm{max}}(N).
\end{equation}

The following lemma shows how to bound the $\Hinf$-norm of the system $\Sigma_{\tau}$ for general covariance matrices $\Omega$ and $\Gamma$. Namely, the provided upper and lower bound expressions isolate the terms that depend on the covariance matrices from the ones that depend on other system-theoretic quantities.

\begin{lem}
\label{lemma:covbounds}
Consider the system $\Sigma_{\tau}$ defined in \eqref{eq:edgeagreement1}. Let $A$, $B$, and $C$ denote the state, input, and output matrices of $\Sigma_{\tau}$, respectively, i.e., $A=-L_{e,s}^{\tau}RWR^{T}$, $B=\begin{bmatrix}D_{\tau}^{T}E^{-1}\Omega & -L_{e,s}^{\tau}R\Gamma\end{bmatrix}$, and $C=R^{T}$. The $\Hinf$-norm of this system satisfies
\begin{align*}
\|\Sigma_{\tau}\|_{\infty}^{2}&\geq\big(\lambda_{\rm{min}}(Q)\lambda_{\rm{min}}(B_{\tau}^{T}B_{\tau})\\&+\lambda_{\rm{min}}(F)\lambda_{\rm{min}}(B_{c}^{T}B_{c})\big)\lambda_{\rm{max}}(J),\\
\|\Sigma_{\tau}\|_{\infty}^{2}&\leq\big(\lambda_{\rm{max}}(Q)\lambda_{\rm{max}}(B_{\tau}^{T}B_{\tau})\\&+\lambda_{\rm{max}}(F)\lambda_{\rm{max}}(B_{c}^{T}B_{c})\big)\lambda_{\rm{max}}(J),
\end{align*}
where $B_{\tau}=E^{-1}D_{\tau}$, $B_{c}=R^{T}L_{e,s}^{\tau}$, $Q=\Omega\Omega^{T}$, $F=\Gamma\Gamma^{T}$, and $J=A^{-T}C^{T}CA^{-1}$.
\end{lem}
\begin{proof}
By Lemma \ref{lemma:zelazo2}, $\|\Sigma_{\tau}\|_{\infty}=\bar{\sigma}\big(\Sigma_{\tau}(s)|_{s=0}\big)$. Then, 
\begin{align*}
\|\Sigma_{\tau}\|_{\infty}^{2}&=\lambda_{\rm{max}}\big(\Sigma_{\tau}(s)\Sigma_{\tau}(s)^{*}|_{s=0}\big)\\
&=\lambda_{\rm{max}}\big(C(-A)^{-1}BB^{T}(-A^{T})^{-1}C^{T}\big)\\
&=\lambda_{\rm{max}}\big((-A^{T})^{-1}C^{T}C(-A)^{-1}BB^{T}\big)\\
&=\lambda_{\rm{max}}\big(A^{-T}C^{T}CA^{-1}BB^{T}\big)=\lambda_{\rm{max}}\big(JBB^{T}\big), 
\end{align*}
where the equality on the third line above follows from \cite[Proposition 3.17]{geir}. We use the obtained expression to find bounds on $\|\Sigma_{\tau}\|_{\infty}^{2}$. Using \eqref{eq:bounds}, it follows that
\begin{equation}
\label{pf:step1}
\lambda_{\rm{max}}(J)\lambda_{\rm{min}}(BB^{T})\leq\lambda_{\rm{max}}(JBB^{T})\leq
\lambda_{\rm{max}}(J)\lambda_{\rm{max}}(BB^{T}).
\end{equation}
$BB^{T}$ can be expressed as $BB^{T}=B_{\tau}^{T}QB_{\tau}+B_{c}^{T}FB_{c}$.
Using Weyl's Theorem \cite[Theorem 4.3.1]{JohnsonHorn_MatrixAnalysis}, it follows that
\begin{align}
\lambda_{\rm{min}}(BB^{T})&\geq\lambda_{\rm{min}}(B_{\tau}^{T}QB_{\tau})+\lambda_{\rm{min}}(B_{c}^{T}FB_{c}),\label{pf:weyl1}\\
\lambda_{\rm{max}}(BB^{T})&\leq\lambda_{\rm{max}}(B_{\tau}^{T}QB_{\tau})+\lambda_{\rm{max}}(B_{c}^{T}FB_{c}).\label{pf:weyl2}
\end{align}
Given the properties of the matrices $E^{-1}$, $D_{\tau}$, $L_{e,s}^{\tau}$, and $R^{T}$, it is not difficult to see that $\mathcal{N}(B_{\tau})=\mathbf{0}$ and $\mathcal{N}(B_{c})=\mathbf{0}$. 
By \cite[Lemma 2]{Foight2019a}, since $Q\in\mathbb{R}^{n\times n}$ is a symmetric matrix and $B_{\tau}\in\mathbb{R}^{n\times n-1}$ has $\mathcal{N}(B_{\tau})=\mathbf{0}$, it follows that
\begin{align*}
\lambda_{\rm{min}}(B_{\tau}^{T}QB_{\tau})&\geq\lambda_{\rm{min}}(Q)\lambda_{\rm{min}}(B_{\tau}^{T}B_{\tau}),\\
\lambda_{\rm{max}}(B_{\tau}^{T}QB_{\tau})&\leq\lambda_{\rm{max}}(Q)\lambda_{\rm{max}}(B_{\tau}^{T}B_{\tau}).
\end{align*}
Similar bounds on the eigenvalues of $B_{c}^{T}FB_{c}$ can also be derived since $F\in\mathbb{R}^{|\mathcal{E}|\times |\mathcal{E}|}$ is a symmetric matrix and $B_{c}\in\mathbb{R}^{|\mathcal{E}|\times n-1}$ has $\mathcal{N}(B_{c})=\mathbf{0}$.
From \eqref{pf:weyl1} and \eqref{pf:weyl2} and these derived bounds, we obtain
\begin{align*}
\lambda_{\rm{min}}(BB^{T})&\geq\lambda_{\rm{min}}(Q)\lambda_{\rm{min}}(B_{\tau}^{T}B_{\tau})+\lambda_{\rm{min}}(F)\lambda_{\rm{min}}(B_{c}^{T}B_{c}),\\
\lambda_{\rm{max}}(BB^{T})&\leq\lambda_{\rm{max}}(Q)\lambda_{\rm{max}}(B_{\tau}^{T}B_{\tau})+\lambda_{\rm{max}}(F)\lambda_{\rm{max}}(B_{c}^{T}B_{c}).
\end{align*}
The proof is concluded by combining the last two inequalities with \eqref{pf:step1}.
\end{proof}

Consider the pairs of covariance matrices $(\Omega_{0},\Gamma_{0})$ and $(\Omega_{1},\Gamma_{1})$ such that $\Omega_{0}$ and $\Gamma_{0}$ share the same maximum and minimum eigenvalues with $\Omega_{1}$ and $\Gamma_{1}$, respectively. By Lemma \ref{lemma:covbounds}, the $\Hinf$-norms $\|\Sigma_{\tau,0}\|_{\infty}$ and $\|\Sigma_{\tau,1}\|_{\infty}$ that correspond to the systems with the particular choice of covariance matrices $(\Omega_{0},\Gamma_{0})$ and $(\Omega_{1},\Gamma_{1})$, respectively, will be governed by the same bounds. Equipped with the bounds computed in Lemma \ref{lemma:covbounds}, the remainder of the paper focuses on the special choice of covariance matrices $\Omega=\sigma_{w}E^{\frac{1}{2}}$ and $\Gamma=\sigma_{v}W^{\frac{1}{2}}$, as considered in \cite{Foight2019ts,Foight2019a}.
\begin{assm}
\label{assump}
The covariance matrices are defined as $\Omega=\sigma_{w}E^{\frac{1}{2}}$ and $\Gamma=\sigma_{v}W^{\frac{1}{2}}$.
\end{assm}
As per Theorem \ref{thm:bounds1}, the special choice of covariance matrices given in Assumption \ref{assump} allows one to compute alternative bound expressions for the $\Hinf$-norm of the system. In \S\ref{sec:optimization}, we make use of these expressions to derive new insights on $\Hinf$-norm minimization and formulate an optimization problem for the selection of edge weights and time scales. Furthermore, Lemma \ref{lemma:covbounds} provides an \emph{a priori} bound on the error resulting from estimating the $\Hinf$-norm of the system with general covariance matrices by that of a system with the special choice of covariance matrices that share the same aforementioned eigenvalue properties. This observation motivates the adoption of a heuristic in the optimization problem of \S\ref{sec:optimization} to tighten the alternative bounds. 

Let $\Omega=\sigma_{w}E^{\frac{1}{2}}$ and $\Gamma=\sigma_{v}W^{\frac{1}{2}}$ as per Assumption \ref{assump}. The resulting system is denoted by $\Tilde{\Sigma}_{\tau}$, and its transfer function matrix is given by
\begin{align}
\label{eq:tfsigma_withcov}
&\Tilde{\Sigma}_{\tau}(s)\\&=R^{T}(sI+L_{e,s}^{\tau}RWR^{T})^{-1}\begin{bmatrix}\sigma_{w}D_{\tau}^{T}E^{-\frac{1}{2}} & -\sigma_{v}L_{e,s}^{\tau}W^{\frac{1}{2}}\end{bmatrix}.
\end{align}

\begin{thm}
\label{thm:hinfexpression}
The $\Hinf$-norm of the system $\Tilde{\Sigma}_{\tau}$ defined in \eqref{eq:tfsigma_withcov} satisfies $\|\Tilde{\Sigma}_{\tau}\|_{\infty}^{2}=\bar{\sigma}(Z)$, where
\begin{equation}
\label{eq:Z}
Z=\sigma_{w}^{2}R^{T}(RWR^{T}L_{e,s}^{\tau}RWR^{T})^{-1}R
+\sigma_{v}^{2}R^{T}(RWR^{T})^{-1}R.
\end{equation}
\end{thm}
\begin{proof}
By Lemma \ref{lemma:zelazo2}, $\|\Tilde{\Sigma}_{\tau}\|_{\infty}^{2}=\lambda_{\rm{max}}\big(\Tilde{\Sigma}_{\tau}(s)\Tilde{\Sigma}_{\tau}(s)^{*}|_{s=0}\big)$, from which the result follows immediately.
\end{proof}

As mentioned in \S\ref{sec:problemsetup}, the edge consensus model $\Sigma_{\tau}$ defined in \eqref{eq:edgeagreement1} (and hence, $\Tilde{\Sigma}_{\tau}$ defined in \eqref{eq:tfsigma_withcov}) corresponds to the edge states of the given spanning tree $\G_{\tau}$ of the underlying system graph $\G$.
Since our problem setup allows for arbitrary time-scaled agents and weighted interconnections, in general, different choices of the spanning tree $\G_{\tau}$ yield different values of the $\Hinf$-norm of the corresponding system $\Tilde{\Sigma}_{\tau}$. We illustrate this observation by considering the graph $\G$ shown in Figure \ref{fig:graphJandtrees}. This graph has $\mathbf{det}(RR^T)=3$ spanning tree subgraphs \cite{Zelazo2011a}, denoted by $\G_{\tau_1}$, $\G_{\tau_2}$, and $\G_{\tau_3}$, and are also shown in Figure \ref{fig:graphJandtrees}. We assign the edge weights as $w_1=4$, $w_2=8$, and $w_3=12$ and the time scales $\epsilon_1=0.1$, $\epsilon_2=0.4$, and $\epsilon_3=0.8$. We further assume that $\sigma_w=\sigma_v=1$. Then, the values of the $\Hinf$-norm of the corresponding system $\Tilde{\Sigma}_{\tau}$ when the spanning trees $\G_{\tau_1}$, $\G_{\tau_2}$, and $\G_{\tau_3}$ are considered are $\|\Tilde{\Sigma}_{\tau_1}\|_{\infty}=0.4230$, $\|\Tilde{\Sigma}_{\tau_2}\|_{\infty}=0.4211$, and $\|\Tilde{\Sigma}_{\tau_3}\|_{\infty}=0.4237$, respectively. These $\Hinf$-norm values are computed using the expression in Theorem \ref{thm:hinfexpression} and are further verified using the built-in MATLAB command \texttt{hinfnorm}. With this in mind, the remainder of the results assume a given choice of the spanning tree for the computation of $\|\Tilde{\Sigma}_{\tau}\|_{\infty}$ and the corresponding bounds. 

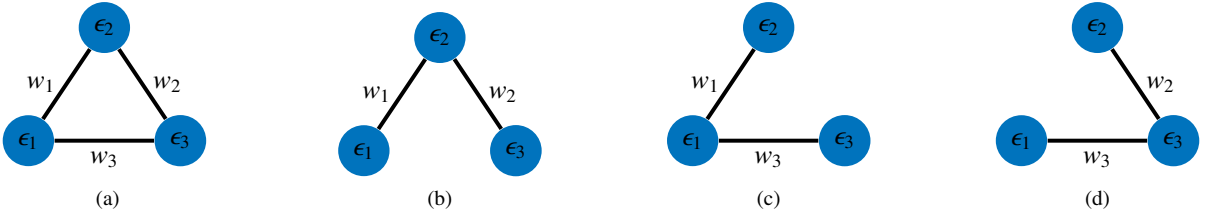
\begin{figure*}
\centering
\subfloat[]{
\begin{tikzpicture}
\node[shape=circle,fill=bluee] (A) at (0,0) {$\epsilon_{1}$};
\node[shape=circle,fill=bluee] (B) at (1,1.5) {$\epsilon_{2}$};
\node[shape=circle,fill=bluee] (C) at (2,0) {$\epsilon_{3}$};

\path [draw,line width=1.5pt,-,black!100] (A) edge node[left] {$w_{1}$} (B);
\path [draw,line width=1.5pt,-,black!100] (B) edge node[right] {$w_{2}$} (C);
\path [draw,line width=1.5pt,-,black!100] (C) edge node[below] {$w_{3}$} (A);
\end{tikzpicture}
\label{graphH}
}
\hfill
\subfloat[]{
\begin{tikzpicture}
\node[shape=circle,fill=bluee] (A) at (0,0) {$\epsilon_{1}$};
\node[shape=circle,fill=bluee] (B) at (1,1.5) {$\epsilon_{2}$};
\node[shape=circle,fill=bluee] (C) at (2,0) {$\epsilon_{3}$};

\path [draw,line width=1.5pt,-,black!100] (A) edge node[left] {$w_{1}$} (B);
\path [draw,line width=1.5pt,-,black!100] (B) edge node[right] {$w_{2}$} (C);
\end{tikzpicture}
\label{tree1}
}
\hfill
\subfloat[]{
\begin{tikzpicture}
\node[shape=circle,fill=bluee] (A) at (0,0) {$\epsilon_{1}$};
\node[shape=circle,fill=bluee] (B) at (1,1.5) {$\epsilon_{2}$};
\node[shape=circle,fill=bluee] (C) at (2,0) {$\epsilon_{3}$};

\path [draw,line width=1.5pt,-,black!100] (A) edge node[left] {$w_{1}$} (B);
\path [draw,line width=1.5pt,-,black!100] (C) edge node[below] {$w_{3}$} (A);
\end{tikzpicture}
\label{tree2}
}
\hfill
\subfloat[]{
\begin{tikzpicture}
\node[shape=circle,fill=bluee] (A) at (0,0) {$\epsilon_{1}$};
\node[shape=circle,fill=bluee] (B) at (1,1.5) {$\epsilon_{2}$};
\node[shape=circle,fill=bluee] (C) at (2,0) {$\epsilon_{3}$};

\path [draw,line width=1.5pt,-,black!100] (B) edge node[right] {$w_{2}$} (C);
\path [draw,line width=1.5pt,-,black!100] (C) edge node[below] {$w_{3}$} (A);
\end{tikzpicture}
\label{tree3}
}
\caption{(a) Graph $\mathcal{\G}$ that consists of $3$ nodes and $3$ edges and its three spanning tree subgraphs (b) $\mathcal{\G}_{\tau_1}$, (c) $\mathcal{\G}_{\tau_2}$, and (d) $\mathcal{\G}_{\tau_3}$. A time scale $\epsilon_{i}$ is associated with every node $i\in\mathcal{V}$, and a weight $w_{l}$ is associated with every edge $l\in\{1,\ldots,|\mathcal{E}|\}$.}
\label{fig:graphJandtrees}
\end{figure*}

To find lower and upper bounds on $\|\Tilde{\Sigma}_{\tau}\|_{\infty}$ obtained in Theorem \ref{thm:hinfexpression}, we consider a modified system $\Pi_{\tau}$ defined by
\begin{equation}
\label{eq:systempi}
\Pi_{\tau}(s)=W^{\frac{1}{2}}\Tilde{\Sigma}_{\tau}(s).
\end{equation} 
System $\Pi_{\tau}$ is defined similarly to system $\Tilde{\Sigma}_{\tau}$, with the difference that the output matrix is weighted by $W^{\frac{1}{2}}$.
\begin{lem}
\label{lemma:HinfPi}
The $\Hinf$-norm of the system $\Pi_{\tau}$ defined in \eqref{eq:systempi} satisfies $\|\Pi_{\tau}\|_{\infty}^{2}=\bar{\sigma}\big(\sigma_{w}^{2}X+\sigma_{v}^{2}Y\big)$, where
\begin{align}
X&=W^{\frac{1}{2}}R^{T}(RWR^{T}L_{e,s}^{\tau}RWR^{T})^{-1}RW^{\frac{1}{2}}, \label{eq:X}\\
Y&=W^{\frac{1}{2}}R^{T}(RWR^{T})^{-1}RW^{\frac{1}{2}}. \label{eq:Y}
\end{align}
\end{lem}
\begin{proof}
A result similar to Lemma \ref{lemma:zelazo2} can be derived to show that $\|\Pi_{\tau}\|_{\infty}=\bar{\sigma}\big(W^{\frac{1}{2}}R^{T}TH(0)\big)$, i.e., introducing the output equation $W^{\frac{1}{2}}R^{T}T\hat{x}_{\tau}(t)$ does not affect the frequency at which the supremum occurs. Thus, the desired expression follows from $\|\Pi_{\tau}\|_{\infty}^{2}=\lambda_{\rm{max}}\big(\Pi_{\tau}(s)\Pi_{\tau}(s)^{*}|_{s=0}\big)$.
\end{proof}
The expression in Lemma \ref{lemma:HinfPi} can be further simplified as explained next. 

\begin{lem}
\label{lemma:weyl}
The matrices $X$ and $Y$ defined in \eqref{eq:X} and \eqref{eq:Y}, respectively, satisfy
\begin{equation*}
\lambda_{\rm{max}}(X+Y)=\lambda_{\rm{max}}(X)+\lambda_{\rm{max}}(Y).
\end{equation*}
\end{lem}
\begin{proof}
Let $q\neq \mathbf{0}$ be the eigenvector of $X$ corresponding to $\lambda_{\rm{max}}(X)$. Then, $Xq=\lambda_{\rm{max}}(X)q\neq \mathbf{0}$, and consequently $q\not\in\mathcal{N}(X)$. That $\lambda_{\rm{max}}(X)q\neq \mathbf{0}$ follows from the fact that $\lambda_{\rm{max}}(X)\neq 0$. Namely, $X=X^T\succeq0$, and so $\mathbf{trace}(X)=0$ if and only if $X$ is the zero matrix \cite[Corollary 7.1.5]{JohnsonHorn_MatrixAnalysis}.
Since $R^{T}$ has full column rank, $\mathcal{N}(X)=\mathcal{N}(Y)=\mathcal{N}(RW^{\frac{1}{2}})$, and so we deduce that $q\not\in\mathcal{N}(Y)$. Moreover, since X and Y commute, i.e., $XY=YX$, and are both diagonalizable, they share the same eigenvector matrix \cite[Chapter 5]{strang_linalgebra}. Therefore, $q$ is an eigenvector of $Y$, and we can write $Yq=yq$, where $y$ is an eigenvalue of $Y$. Since $q\not\in\mathcal{N}(Y)$, then $Yq=yq\neq\mathbf{0}$ and $y\neq 0$. It can be verified that $Y$ is a projection matrix, i.e., $Y^{2}=Y$, and so it has $n-1$ eigenvalues at $1$ and the remaining eigenvalues at zero. Hence, $Yq=\lambda_{\rm{max}}(Y)q=q$. 

Finally, $(X+Y)q=\big(\lambda_{\rm{max}}(X)+\lambda_{\rm{max}}(Y)\big)q$. Given that $q\neq \mathbf{0}$ and $\lambda_{\rm{max}}(X)+\lambda_{\rm{max}}(Y)>0$, then $\lambda_{\rm{max}}(X)+\lambda_{\rm{max}}(Y)$ is an eigenvalue of $X+Y$. By Weyl's Theorem \cite[Theorem 4.3.1]{JohnsonHorn_MatrixAnalysis}, $\lambda_{\rm{max}}(X+Y)\leq \lambda_{\rm{max}}(X)+\lambda_{\rm{max}}(Y)$. Hence, this inequality is binding, and $\lambda_{\rm{max}}(X+Y)= \lambda_{\rm{max}}(X)+\lambda_{\rm{max}}(Y)$.
\end{proof}

\begin{thm} 
\label{thm:hinfpi}
The $\Hinf$-norm of the system $\Pi_{\tau}$ defined in \eqref{eq:systempi} satisfies
\begin{equation}
\label{eq:HinfPi}
\|\Pi_{\tau}\|_{\infty}^{2}=\sigma_{w}^{2}\bar{\sigma}(X)+\sigma_{v}^{2},
\end{equation}
where $X$ is defined in \eqref{eq:X}.
\end{thm}
\begin{proof}
From Lemma \ref{lemma:HinfPi}, $\|\Pi_{\tau}\|_{\infty}^{2}=\lambda_{\rm{max}}(\sigma_{w}^{2}X+\sigma_{v}^{2}Y)$.
Applying a slightly modified version of Lemma \ref{lemma:weyl}, we get
\begin{equation*}
\|\Pi_{\tau}\|_{\infty}^{2}=\lambda_{\rm{max}}(\sigma_{w}^{2}X)+\lambda_{\rm{max}}(\sigma_{v}^{2}Y)=\sigma_{w}^{2}\lambda_{\rm{max}}(X)+\sigma_{v}^{2},
\end{equation*}
where the final equality follows from the fact that $Y$ is a projection matrix, i.e., $\lambda_{\rm{max}}(Y)=1$.
\end{proof}
By Theorem \ref{thm:hinfpi}, it can be seen that the $\Hinf$-norm of the system $\Pi_{\tau}$ only requires the computation of the largest eigenvalue of $X$ defined in \eqref{eq:X}. Further simplifications can be performed on graphs with equal edge weights, i.e., $W=\rho I$. The case $\rho=1$ corresponds to unweighted graphs.
\begin{crly}
\label{crly:HinfWithoutWeights}
Consider the system $\Tilde{\Sigma}_{\tau}$ defined in \eqref{eq:tfsigma_withcov}. Assume that all the edge weights are equal, i.e., $W=\rho I$ for some $\rho>0$. Then, the $\Hinf$-norm of the system $\Tilde{\Sigma}_{\tau}$ satisfies
\begin{equation*}
\|\Tilde{\Sigma}_{\tau}\|_{\infty}^{2}=\frac{1}{\rho^2}\sigma_{w}^{2}\bar{\sigma}(R^{T}(RR^{T}L_{e,s}^{\tau}RR^{T})^{-1}R)+\frac{1}{\rho}\sigma_{v}^{2}.
\end{equation*}
\end{crly}
\begin{proof}
For the special case wherein $W=\rho I$, it follows that $\Pi_{\tau}(s)=\sqrt{\rho}\Tilde{\Sigma}_{\tau}(s)$, and so $\|\Tilde{\Sigma}_{\tau}\|_{\infty}^{2}=\frac{1}{\rho}\|\Pi_{\tau}\|_{\infty}^{2}$, from which the desired result follows immediately.
\end{proof}
\begin{rmk}
For the special case of $W=\rho I$, the second term in $Z$ defined in \eqref{eq:Z}, i.e., $R^{T}(RR^{T})^{-1}R$, is a projection matrix \cite[Theorem 3.7]{Zelazo2011a}. Hence, a slightly modified version of Lemma \ref{lemma:weyl} can be applied to simplify the expression for $\|\Tilde{\Sigma}_{\tau}\|_{\infty}^{2}$ in Theorem \ref{thm:hinfexpression} directly. For a general weight matrix $W$, it is not always true that $R^{T}(RWR^{T})^{-1}R$ is a projection matrix. For this reason, the system $\Pi_{\tau}$ is considered, wherein the corresponding term $Y$ is a projection matrix.
\end{rmk}

The expression for $\|\Pi_{\tau}\|_{\infty}^{2}$ obtained in \eqref{eq:HinfPi} can be used to calculate new upper and lower bounds on the $\Hinf$-norm of the original system defined in \eqref{eq:tfsigma_withcov}.
\begin{thm} Consider systems $\tilde{\Sigma}_{\tau}$ and $\Pi_\tau$ defined in \eqref{eq:tfsigma_withcov} and \eqref{eq:systempi}, respectively. The $\Hinf$-norm of the system $\Tilde{\Sigma}_{\tau}$ satisfies
\label{thm:bounds1}
\begin{equation*}
\frac{\|\Pi_{\tau}\|_{\infty}}{\lambda_{\rm{max}}(W^{\frac{1}{2}})}\leq\|\Tilde{\Sigma}_{\tau}\|_{\infty}\leq \frac{\|\Pi_{\tau}\|_{\infty}}{\lambda_{\rm{min}}(W^{\frac{1}{2}})}.
\end{equation*}
\end{thm}
\begin{proof}
We aim to find bounds on $\|\Tilde{\Sigma}_{\tau}\|_{\infty}^{2}=\lambda_{\rm{max}}(Z)$ as a function of $\|\Pi_{\tau}\|_{\infty}^{2}=\lambda_{\rm{max}}(V)$, where $V=\sigma^{2}_{w}X+\sigma_{v}^{2}Y$,
$Z$ is defined in \eqref{eq:Z}, and $X$ and $Y$ are defined in \eqref{eq:X} and \eqref{eq:Y}, respectively. Since $V=W^{\frac{1}{2}}ZW^{\frac{1}{2}}$, it follows from the upper bound in \eqref{eq:bounds} that
\begin{align*}
\lambda_{\rm{max}}(W^{\frac{1}{2}}ZW^{\frac{1}{2}})&\leq\lambda_{\rm{max}}(W^{\frac{1}{2}}Z)\lambda_{\rm{max}}(W^{\frac{1}{2}})\\&\leq \lambda_{\rm{max}}(Z)\lambda_{\rm{max}}(W^{\frac{1}{2}})^{2},
\end{align*}
from which we obtain $\lambda_{\rm{max}}(Z) \geq\lambda_{\rm{max}}(V)/\lambda_{\rm{max}}(W^{\frac{1}{2}})^{2}$. Moreover, from the lower bound in \eqref{eq:bounds}, it follows that
\begin{align*}
\lambda_{\rm{max}}(W^{\frac{1}{2}}ZW^{\frac{1}{2}}) &\geq \lambda_{\rm{max}}(W^{\frac{1}{2}}Z)\lambda_{\rm{min}}(W^{\frac{1}{2}})\\
&\geq  \lambda_{\rm{max}}(ZW^{\frac{1}{2}})\lambda_{\rm{min}}(W^{\frac{1}{2}})\\
 &\geq \lambda_{\rm{max}}(Z)\lambda_{\rm{min}}(W^{\frac{1}{2}})^{2}.
\end{align*}
Hence, we get that $\lambda_{\rm{max}}(Z)\leq\lambda_{\rm{max}}(V)/\lambda_{\rm{min}}(W^{\frac{1}{2}})^{2}$.
\end{proof}
We note that the ratio of the upper bound to the lower bound on $\|\Tilde{\Sigma}_{\tau}\|_{\infty}$, obtained in Theorem \ref{thm:bounds1}, is a function of the largest and smallest edge weights only. Namely, this ratio is equal to
\begin{equation}
\label{eq:ratio}
\eta=\frac{\lambda_{\rm{max}}(W^{\frac{1}{2}})}{\lambda_{\rm{min}}(W^{\frac{1}{2}})}.
\end{equation}
Therefore, when all edge weights are equal to each other, $\|\Tilde{\Sigma}_{\tau}\|_{\infty}$ is equal to both its lower and upper bounds. In this case, the expression for $\|\Tilde{\Sigma}_{\tau}\|_{\infty}$ is given in Corollary \ref{crly:HinfWithoutWeights}. In addition to allowing for $W=\rho I$, where $\rho$ may be different from $1$, a novel contribution of Corollary \ref{crly:HinfWithoutWeights} is that the expression derived therein holds for arbitrary node time scales. Moreover, if all the edge weights and time scales are equal to unity, i.e., $W=I$ and $E=I$, then the expression for $\|\Tilde{\Sigma}_{\tau}\|_{\infty}^{2}$ given in \cite[Theorem 3.7]{Zelazo2011a} is retrieved. To tighten the bounds on the $\Hinf$-norm of $\Tilde{\Sigma}_{\tau}$ in the case $W\neq \rho I$, we propose a heuristic in the optimization setup of \S\ref{sec:optimization}, by which we impose a bound on the upper to lower bound ratio $\eta$ defined in \eqref{eq:ratio}.

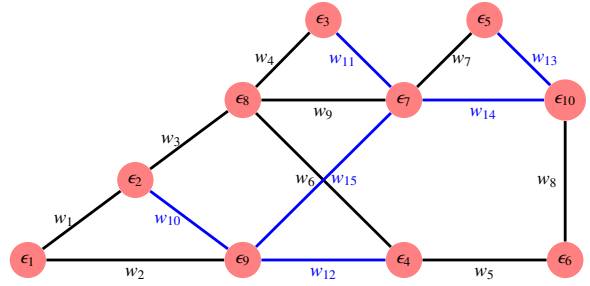
\begin{figure}
\centering
\resizebox{\columnwidth}{!}{
\begin{tikzpicture}
    \node[shape=circle,fill=red!50] (B) at (0,0) {$\epsilon_{2}$};
    \node[shape=circle,fill=red!50] (H) at (2,1.5) {$\epsilon_{8}$};
    \node[shape=circle,fill=red!50] (G) at (5,1.5) {$\epsilon_{7}$};
    \node[shape=circle,fill=red!50] (J) at (8,1.5) {$\epsilon_{10}$};
    \node[shape=circle,fill=red!50] (C) at (3.5,3) {$\epsilon_{3}$};
    \node[shape=circle,fill=red!50] (E) at (6.5,3) {$\epsilon_{5}$};
    \node[shape=circle,fill=red!50] (I) at (2,-1.5) {$\epsilon_{9}$};
    \node[shape=circle,fill=red!50] (D) at (5,-1.5) {$\epsilon_{4}$};
    \node[shape=circle,fill=red!50] (F) at (8,-1.5) {$\epsilon_{6}$};
    \node[shape=circle,fill=red!50] (A) at (-2,-1.5) {$\epsilon_{1}$};

    \path [draw,line width=1.5pt,-,black!100] (A) edge node[left] {$w_{1}$} (B);
    \path [draw,line width=1.5pt,-,black!100] (A) edge node[below] {$w_{2}$} (I);
    \path [draw,line width=1.5pt,-,black!100] (B) edge node[left] {$w_{3}$} (H);
    \path [draw,line width=1.5pt,-,black!100] (C) edge node[left] {$w_{4}$} (H);
    \path [draw,line width=1.5pt,-,black!100] (D) edge node[below] {$w_{5}$} (F);
    \path [draw,line width=1.5pt,-,black!100] (D) edge node[left] {$w_{6}$} (H);
    \path [draw,line width=1.5pt,-,black!100] (E) edge node[right] {$w_{7}$} (G);
    \path [draw,line width=1.5pt,-,black!100] (F) edge node[left] {$w_{8}$} (J);
    \path [draw,line width=1.5pt,-,black!100] (G) edge node[below] {$w_{9}$} (H);
    \path [draw,line width=1.5pt,-,blue!100] (I) edge node[left] {$w_{10}$} (B);
    \path [draw,line width=1.5pt,-,blue!100] (C) edge node[left] {$w_{11}$} (G);
    \path [draw,line width=1.5pt,-,blue!100] (D) edge node[below] {$w_{12}$} (I);
    \path [draw,line width=1.5pt,-,blue!100] (E) edge node[right] {$w_{13}$} (J);
    \path [draw,line width=1.5pt,-,blue!100] (G) edge node[below] {$w_{14}$} (J);
    \path [draw,line width=1.5pt,-,blue!100] (G) edge node[right] {$w_{15}$} (I);
\end{tikzpicture}}
\caption{Graph $\G$ that consists of $10$ nodes and $15$ edges. A time scale $\epsilon_{i}$ is associated with every node $i\in\mathcal{V}$, and a weight $w_{l}$ is associated with every edge $l\in\{1,\ldots,|\mathcal{E}|\}$. The edges colored in black are the edges of the chosen spanning tree, while the edges colored in blue are the corresponding co-tree edges.}
\label{fig:graph}
\end{figure}

To illustrate the bounds obtained in Theorem \ref{thm:bounds1}, consider the graph $\G$ shown in Figure \ref{fig:graph}. This graph consists of $n=10$ nodes and was randomly generated using an edge probability of $\frac{\ln n}{n}=0.23$ \cite[Theorem 7.3]{randomgraphs}. The $\Hinf$-norm of the corresponding system $\Tilde{\Sigma}_{\tau}$ and the bounds obtained in Theorem \ref{thm:bounds1} are computed for different combinations of edge weights and time scales. The results are presented in Figure \ref{fig:bounds1}, wherein it is assumed that $\sigma_{w}=\sigma_{v}=1$. The examples in Figure \ref{fig:bounds1} are generated as follows. In example $1$, the time scale and edge weight matrices are $E=I$ and $W=I$. In examples $2-4$, $W=I$, and $E$ is varied. In examples $5-8$, $E=I$, and $W$ is varied. Finally, in examples $9-14$, both $W$ and $E$ are varied. We denote the upper and lower bounds obtained in Theorem \ref{thm:bounds1} by UB and LB, respectively. As expected from Corollary \ref{crly:HinfWithoutWeights}, the ratios UB/$\|\Tilde{\Sigma}_{\tau}\|_{\infty}=$ LB/$\|\Tilde{\Sigma}_{\tau}\|_{\infty}=1$ in examples $1-4$. On the other hand, in examples $5-14$, UB/$\|\Tilde{\Sigma}_{\tau}\|_{\infty}\geq 1$ and LB/$\|\Tilde{\Sigma}_{\tau}\|_{\infty}\leq 1$, as expected from Theorem \ref{thm:bounds1}. 


\begin{figure}
    \centering
    \includegraphics[width=\columnwidth]{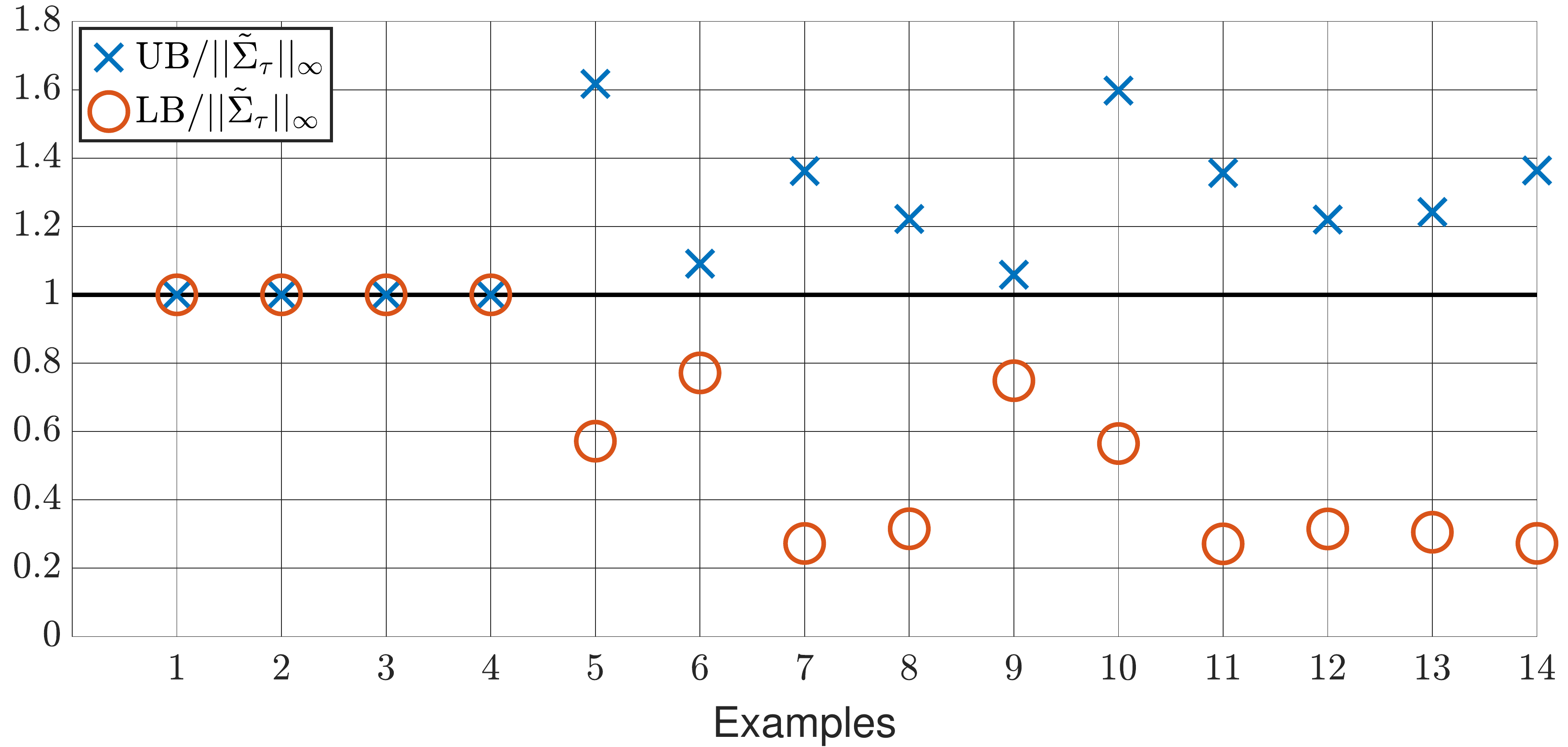}
    \caption{Ratios UB/$\|\Tilde{\Sigma}_{\tau}\|_{\infty}$ and LB/$\|\Tilde{\Sigma}_{\tau}\|_{\infty}$ for different edge weights and time scales combinations. UB and LB are the upper and lower bounds, respectively, obtained in Theorem \ref{thm:bounds1}.} 
    \label{fig:bounds1}
\end{figure}

\S\ref{sec:Hinf} is concluded by deriving results for the special case when the underlying graph is a spanning tree.
\begin{crly}
The $\Hinf$-norms of the systems $\Tilde{\Sigma}_{\tau}$ and $\Pi_{\tau}$ defined in \eqref{eq:tfsigma_withcov} and \eqref{eq:systempi}, respectively, when the underlying graph is a spanning tree, are given by
\begin{align}
\|\Tilde{\Sigma}_{\tau}\|_{\infty}^{2}&=\bar{\sigma}\big(\sigma_{w}^{2}(WL_{e,s}^{\tau}W)^{-1}+\sigma_{v}^{2}W^{-1}\big), \label{eq:hinfsigmatree}\\
\|\Pi_{\tau}\|_{\infty}^{2}&=\sigma_{w}^{2}\bar{\sigma}\big((W^{\frac{1}{2}}L_{e,s}^{\tau}W^{\frac{1}{2}})^{-1}\big)+\sigma_{v}^{2}.\label{eq:hinfpitree}
\end{align}
\end{crly}
\begin{proof}
\eqref{eq:hinfsigmatree} and \eqref{eq:hinfpitree} are obtained by simply substituting $R=I$ in the $\Hinf$-norm expressions in Theorems \ref{thm:hinfexpression} and \ref{thm:hinfpi}, respectively.
\end{proof}
For the case of spanning trees, we find bounds on the $\Hinf$-norm of $\Tilde{\Sigma}_{\tau}$ that only require the knowledge of the minimum eigenvalue of $L_{e,s}^{\tau}$ and the largest and smallest edge weights. 

\begin{thm}
\label{thm:bounds2}
When the underlying graph is a spanning tree, the $\Hinf$-norm of the system $\Tilde{\Sigma}_{\tau}$ defined in \eqref{eq:tfsigma_withcov} satisfies
$L\leq\|\Tilde{\Sigma}_{\tau}\|_{\infty}^{2}\leq U$, where $L=\max\{L_1,L_2\}$,
\begin{align*}
L_1&=\frac{\sigma_{w}^{2}+\sigma_{v}^{2}\lambda_{\rm{min}}(L_{e,s}^{\tau})\lambda_{\rm{max}}(W^{\frac{1}{2}})^{2}}{\lambda_{\rm{min}}(L_{e,s}^{\tau})\lambda_{\rm{max}}(W^{\frac{1}{2}})^{4}},\\
L_2&=\frac{\sigma_{w}^{2}+\sigma_{v}^{2}\lambda_{\rm{max}}(L_{e,s}^{\tau})\lambda_{\rm{max}}(W^{\frac{1}{2}})\lambda_{\rm{min}}(W^{\frac{1}{2}})}{\lambda_{\rm{max}}(L_{e,s}^{\tau})\lambda_{\rm{max}}(W^{\frac{1}{2}})^{3}\lambda_{\rm{min}}(W^{\frac{1}{2}})},\\
U&=\frac{\sigma_{w}^{2}+\sigma_{v}^{2}\lambda_{\rm{min}}(L_{e,s}^{\tau})\lambda_{\rm{min}}(W^{\frac{1}{2}})^{2}}{\lambda_{\rm{min}}(L_{e,s}^{\tau})\lambda_{\rm{min}}(W^{\frac{1}{2}})^{4}}.
\end{align*}
\end{thm}
\begin{proof}
We first find bounds on $\|\Pi_{\tau}\|_{\infty}^{2}$ obtained in \eqref{eq:hinfpitree} using \eqref{eq:bounds}. Namely, an approach similar to the one used in the proof of Theorem \ref{thm:bounds1} is followed to find bounds on $\lambda_{\rm{max}}\big((W^{\frac{1}{2}}L_{e,s}^{\tau}W^{\frac{1}{2}})^{-1}\big)=\lambda_{\rm{max}}\big(W^{-\frac{1}{2}}(L_{e,s}^{\tau})^{-1}W^{-\frac{1}{2}}\big)$. We obtain the following two alternative lower bound expressions:
\begin{subequations}
\begin{align}
\begin{split}
&\lambda_{\rm{max}}\big((L_{e,s}^{\tau})^{-1}\big)\lambda_{\text{\rm{min}}}(W^{-\frac{1}{2}})^{2}\\&=\big(\lambda_{\rm{min}}(L_{e,s}^{\tau})\lambda_{\rm{max}}(W^{\frac{1}{2}})^{2}\big)^{-1},
\end{split}\label{subeqn:oldLB1}\\
\begin{split}
&\lambda_{\rm{min}}\big((L_{e,s}^{\tau})^{-1}\big)\lambda_{\text{\rm{min}}}(W^{-\frac{1}{2}})\lambda_{\text{\rm{max}}}(W^{-\frac{1}{2}})\\&=\big(\lambda_{\rm{max}}(L_{e,s}^{\tau})\lambda_{\rm{max}}(W^{\frac{1}{2}})\lambda_{\rm{min}}(W^{\frac{1}{2}})\big)^{-1}
\end{split}\label{subeqn:newLB1}.
\end{align}
\end{subequations}
Moreover, the upper bound expression is given by
\begin{equation}
\lambda_{\rm{max}}\big((L_{e,s}^{\tau})^{-1}\big)\lambda_{\rm{max}}(W^{-\frac{1}{2}})^{2}=\big(\lambda_{\rm{min}}(L_{e,s}^{\tau})\lambda_{\rm{min}}(W^{\frac{1}{2}})^{2}\big)^{-1}.\label{eqn:UB}
\end{equation}
In \eqref{subeqn:oldLB1}, \eqref{subeqn:newLB1}, and \eqref{eqn:UB}, the rightmost expressions follow from the fact that $\lambda_{\rm{max}}(P^{-1})=1/\lambda_{\rm{min}}(P)$ for any $P\succ0$.
Thus, it follows from \eqref{subeqn:oldLB1}, \eqref{subeqn:newLB1}, and \eqref{eqn:UB} that the $\Hinf$-norm of the system $\Pi_{\tau}$ satisfies 
\begin{align*}
\|\Pi_{\tau}\|_{\infty}^{2}& \geq \max\left\{\frac{\sigma_{w}^{2}}{\lambda_{\rm{min}}(L_{e,s}^{\tau})\lambda_{\rm{max}}(W^{\frac{1}{2}})^{2}}\right.\\
&\left.+\sigma_{v}^{2},\frac{\sigma_{w}^{2}}{\lambda_{\rm{max}}(L_{e,s}^{\tau})\lambda_{\rm{max}}(W^{\frac{1}{2}})\lambda_{\rm{min}}(W^{\frac{1}{2}})}+\sigma_{v}^{2}\right\},\\
\|\Pi_{\tau}\|_{\infty}^{2}& \leq
 \frac{\sigma_{w}^{2}}{\lambda_{\rm{min}}(L_{e,s}^{\tau})\lambda_{\rm{min}}(W^{\frac{1}{2}})^{2}}+\sigma_{v}^{2}.
\end{align*}
From Theorem \ref{thm:bounds1}, we have that
\begin{equation*}
\frac{\|\Pi_{\tau}\|_{\infty}^{2}}{\lambda_{\rm{max}}(W^{\frac{1}{2}})^{2}}\leq\|\Tilde{\Sigma}_{\tau}\|_{\infty}^{2}\leq \frac{\|\Pi_{\tau}\|_{\infty}^{2}}{\lambda_{\rm{min}}(W^{\frac{1}{2}})^{2}},
\end{equation*}
from which $\max\{L_1,L_2\} \leq \|\Tilde{\Sigma}_{\tau}\|_{\infty}^{2}\leq U$ is obtained by replacing $\|\Pi_{\tau}\|_{\infty}^{2}$ on the right-hand-side of the above inequality by its upper bound and on the left-hand-side by its lower bound. 
\end{proof}
Further simplifications can be performed on spanning tree subgraphs having equal edge weights.
\begin{crly}
\label{crly:equalbounds}
Consider systems $\Tilde{\Sigma}_{\tau}$ and $\|\Pi_{\tau}\|$ defined in \eqref{eq:tfsigma_withcov} and \eqref{eq:systempi}, respectively. Assuming that the underlying graph is a spanning tree and has equal edge weights, i.e., $W=\rho I$ for some $\rho>0$, it follows that
\begin{equation*}
\|\Tilde{\Sigma}_{\tau}\|_{\infty}^{2}=\frac{1}{\rho}\|\Pi_{\tau}\|_{\infty}^{2}=\frac{1}{\rho^2}\sigma_{w}^{2}\bar{\sigma}\big((L_{e,s}^{\tau})^{-1}\big)+\frac{1}{\rho}\sigma_{v}^{2}=L=U,
\end{equation*}
where $L$ and $U$ are defined in Theorem \ref{thm:bounds2}.
\end{crly}
\begin{proof}
The desired result follows immediately by substituting $W=\rho I$ in the $\Hinf$-norm expressions of $\Tilde{\Sigma}_{\tau}$ and $\Pi_{\tau}$ defined in \eqref{eq:hinfsigmatree} and \eqref{eq:hinfpitree}, respectively, and in the expressions of the bounds in Theorems \ref{thm:bounds1} and \ref{thm:bounds2}. We note that in this case $L=L_1$ since 
\begin{align*}
L_1&=\frac{1}{\rho^2}\sigma_{w}^{2}\lambda_{\rm{max}}\big((L_{e,s}^{\tau})^{-1}\big)+\frac{1}{\rho}\sigma_{v}^{2}\\
&\geq\frac{1}{\rho^2}\sigma_{w}^{2}\lambda_{\rm{min}}\big((L_{e,s}^{\tau})^{-1}\big)+\frac{1}{\rho}\sigma_{v}^{2}\\
&=L_2.    
\end{align*}
\end{proof}

Revisiting the graph shown in Figure \ref{fig:graph}, consider the spanning tree subgraph that consists of the edges colored in black. The $\Hinf$-norm of the corresponding system $\Tilde{\Sigma}_{\tau}$ and the bounds obtained in Theorems \ref{thm:bounds1} and \ref{thm:bounds2} are computed for various combinations of $W$ and $E$, similar to the combinations considered in the previous example. The results are presented in Figure \ref{fig:bounds2}, wherein it is assumed that $\sigma_{w}=\sigma_{v}=1$. We denote the upper and lower bounds obtained in Theorem \ref{thm:bounds1} by UB and LB, respectively, and the upper and lower bounds obtained in Theorem \ref{thm:bounds2} by UB1 and LB1, respectively, with LB1 $=\sqrt{L}$ and UB1 $=\sqrt{U}$. As expected from Corollary \ref{crly:equalbounds}, $\|\Tilde{\Sigma}_{\tau}\|_{\infty}=$ UB $=$ LB $=$ UB1$ =$ LB1 in examples $1-4$. On the other hand, in examples $5-14$, UB1 $\geq$ UB and LB1 $\leq$ LB,  as expected from Theorem \ref{thm:bounds2}.


\begin{figure}
    \centering
    \includegraphics[width=\columnwidth]{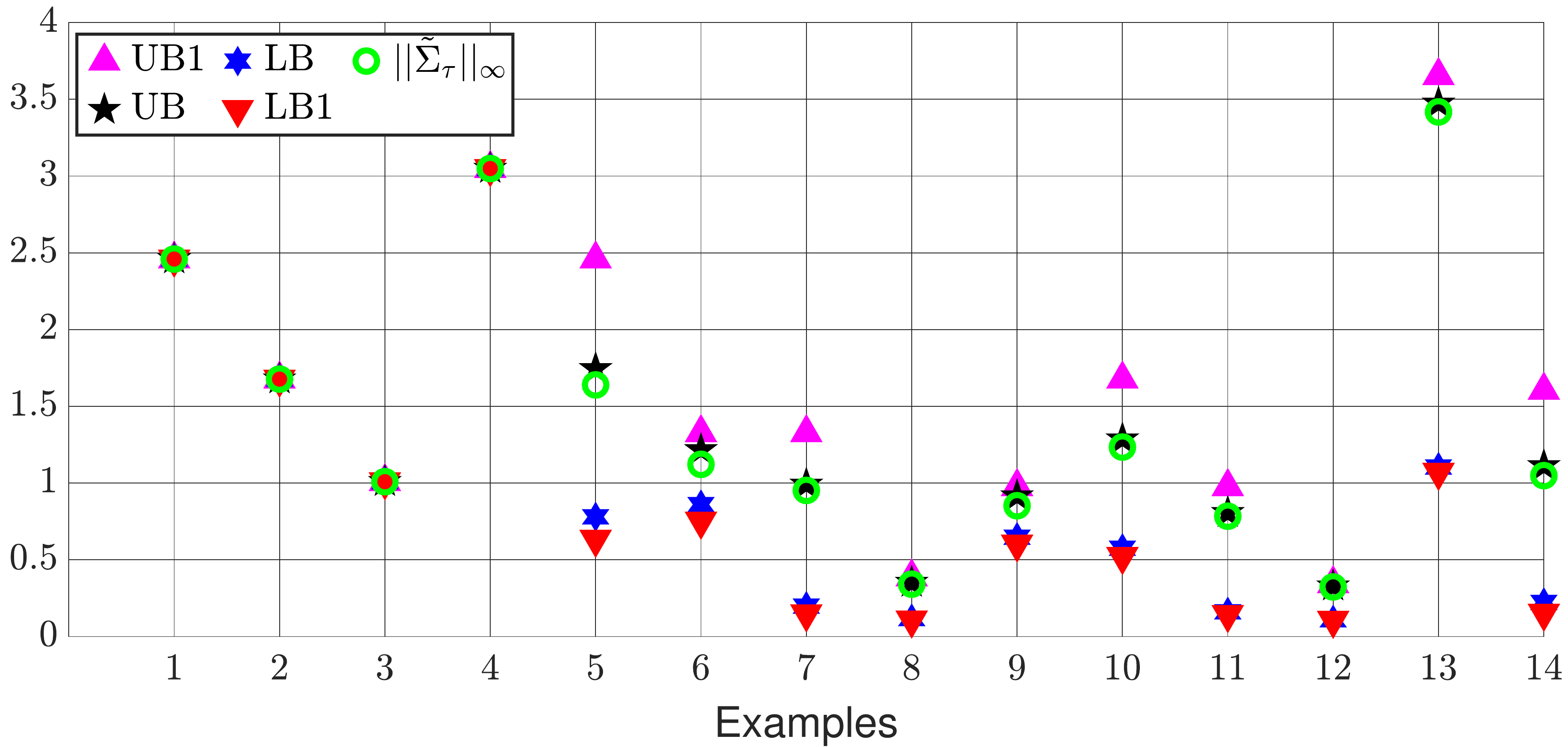}
    \caption{$\|\Tilde{\Sigma}_{\tau}\|_{\infty}$, UB, LB, UB1, and LB1 for different edge weights and time scales combinations. UB and LB are the upper and lower bounds, respectively, obtained in Theorem \ref{thm:bounds1}. UB1 and LB1 are the upper and lower bounds, respectively, obtained in Theorem \ref{thm:bounds2}.} 
    \label{fig:bounds2}
\end{figure}

\subsection{Optimal Time Scales and Edge Weights}
\label{sec:optimization}
In this section, the expression of $\|\Pi_{\tau}\|_{\infty}^{2}$ in (\ref{eq:HinfPi}) is utilized to derive new insights on the $\Hinf$-norm minimization problem and to design an optimization problem for the selection of time scales and/or edge weights for the system $\Tilde{\Sigma}_{\tau}$ defined in (\ref{eq:tfsigma_withcov}). We define the vectors of time scales and edge weights as $\epsilon=(\epsilon_{1},\ldots,\epsilon_{n})$ and $w=(w_{1},\ldots,w_{|\mathcal{E}|})$, respectively. Thus, $E=\mathbf{diag}(\epsilon)$ and $W=\mathbf{diag}(w)$. Moreover, $\epsilon^{-1}$ and $w^{-\frac{1}{2}}$ denote element-wise operations on the vectors $\epsilon$ and $w$, respectively, i.e., $(\epsilon^{-1})_{i}=(\epsilon_i)^{-1}$ for all $i\in\mathcal{V}$ and $(w^{-\frac{1}{2}})_l=(w_l)^{-\frac{1}{2}}$ for all $l\in\{1,\ldots,|\mathcal{E}|\}$. Further assume that upper and lower bounds on the edge weights and time scales are given, i.e., $w_{\rm{min}}\leq w_l\leq w_{\rm{max}}$ and $\epsilon_{\rm{min}}\leq \epsilon_i\leq \epsilon_{\rm{max}}$, where homogeneous bounds are assumed. 
\begin{prop}
\label{prop:trivialsolution}
Consider the system $\Pi_{\tau}(E,W)$ defined in \eqref{eq:systempi} and the bounds $w_{\rm{min}}I\preceq W\preceq w_{\rm{max}}I$ and $\epsilon_{\rm{min}}I\preceq E\preceq\epsilon_{\rm{max}}I$. Then, $E=E_*=\epsilon_{\rm{min}}I$ and $W=W_*=w_{\rm{max}}I$ minimize $\|\Pi_{\tau}(E,W)\|_{\infty}$.
\end{prop}
\begin{proof}
From the expression of $\|\Pi_{\tau}\|_{\infty}^2$ in \eqref{eq:HinfPi}, it can be seen that minimizing $\lambda_{\rm{max}}(X)$ leads to minimizing $\|\Pi_{\tau}\|_{\infty}^2$. From the expression of $X$ in \eqref{eq:X}, it follows that
\begin{align*}
\lambda_{\rm{max}}(X)&=\lambda_{\rm{max}}\big((RWR^{T})^{-\frac{1}{2}}(L_{e,s}^{\tau})^{-1}(RWR^{T})^{-\frac{1}{2}}\big)\\
&=\lambda_{\rm{max}}\big((L_{e,s}^{\tau})^{-\frac{1}{2}}(RWR^{T})^{-1}(L_{e,s}^{\tau})^{-\frac{1}{2}}\big).
\end{align*} From the fact that $0\prec E_{*}\preceq E$ for any admissible $E$, it follows that
\begin{equation*}
(L_{e,s}^{\tau})_{*}^{-1}=(D_{\tau}^{T}E_{*}^{-1}D_{\tau})^{-1}\preceq (D_{\tau}^{T}E^{-1}D_{\tau})^{-1}=(L_{e,s}^{\tau})^{-1}.
\end{equation*} Hence, for a fixed $W$, $\lambda_{\rm{max}}(X)$ is minimized for $E=E_*$.
Similarly, from $W_{*}\succeq W\succ 0$ for any admissible $W$, it follows that
\begin{equation*}
(RW_{*}R^{T})^{-1}\preceq(RWR^{T})^{-1}.
\end{equation*} Hence, for a fixed $E$, $\lambda_{\rm{max}}(X)$ is minimized for $W=W_*$.
\end{proof}
As per Proposition \ref{prop:trivialsolution}, if $E=E_*=\epsilon_{\rm{min}}I$ and $W=W_*=w_{\rm{max}}I$, then $\|\Pi_{\tau}\|_{\infty}$ is minimized, and so is $\|\Tilde{\Sigma}_{\tau}\|_{\infty}=\|\Pi_{\tau}\|_{\infty}/\sqrt{w_{\rm{max}}}$. This is a novel insight in $\Hinf$-based network optimization. Namely, to minimize the $\Hinf$-norm of the edge consensus model in \eqref{eq:tfsigma_withcov}, we operate at minimum time scales and maximum edge weights.

In addition (and in contrast) to the above, we propose the following optimization paradigm if diversity of time scales and edge weights is desirable in the particular application of interest. For generality, we allow for non-homogeneous bounds on the time scales and edge weights, i.e., $w_{{\rm{max},}l}^{-\frac{1}{2}}\leq w_{l}^{-\frac{1}{2}}\leq w_{{\rm{min},}l}^{-\frac{1}{2}}$ and $\epsilon_{{\rm{max},}i}^{-1}\leq\epsilon_{i}^{-1}\leq\epsilon_{{\rm{min},}i}^{-1}$ for all $i\in\mathcal{V}$ and $l\in\{1,\ldots,|\mathcal{E}|\}$. As per the expression of $\|\Pi_{\tau}\|_{\infty}^2$ in \eqref{eq:HinfPi}, it is desirable to minimize $\lambda_{\rm{max}}(X)$. This is done by finding the minimum $\zeta$ such that $X\preceq \zeta I$. To formulate our problem as a convex optimization problem, we minimize $\lambda_{\rm{max}}(X_{1})$ instead, where 
\begin{equation}
X_1=W^{-\frac{1}{2}}R^{\dagger}(L_{e,s}^{\tau})^{-1}(R^{\dagger})^{T}W^{-\frac{1}{2}},
\label{eq:X1}
\end{equation} 
and $\lambda_{\rm{max}}(X_{1})\geq\lambda_{\rm{max}}(X)$ as per Lemma \ref{lemma:opt}.
\begin{lem}
\label{lemma:opt}
The matrices $X$ and $X_1$ defined in \eqref{eq:X} and \eqref{eq:X1}, respectively, satisfy $\lambda_{\rm{max}}(X_{1})\geq\lambda_{\rm{max}}(X)$.
\end{lem}
\begin{proof}
From the expression of $X_1$, it follows that
\begin{equation*}
\lambda_{\rm{max}}(X_1)=\lambda_{\rm{max}}\big((L_{e,s}^{\tau})^{-\frac{1}{2}}R^{\dagger^{T}} W^{-1}R^\dagger(L_{e,s}^{\tau})^{-\frac{1}{2}}\big).
\end{equation*}
From the expression of $\lambda_{\rm{max}}(X)$ in the proof of Proposition \ref{prop:trivialsolution}, it follows that $\lambda_{\rm{max}}(X_{1})\geq\lambda_{\rm{max}}(X)$ is equivalent to
\begin{align*}
&\lambda_{\rm{max}}\big((L_{e,s}^{\tau})^{-\frac{1}{2}} R^{\dagger^{T}} W^{-1}R^\dagger(L_{e,s}^{\tau})^{-\frac{1}{2}}\big)\\&\geq\lambda_{\rm{max}}\big((L_{e,s}^{\tau})^{-\frac{1}{2}}(RWR^{T})^{-1}(L_{e,s}^{\tau})^{-\frac{1}{2}}\big).
\end{align*}
Proving the inequality above can be done by showing that
\begin{equation*}
(R^{\dagger})^{T}W^{-1}R^{\dagger}\succeq (RWR^{T})^{-1}.
\end{equation*}
Expressing $R^{\dagger}$ as $R^{T}(RR^{T})^{-1}$ and applying the Schur complement formula twice, the proof further simplifies to showing that
\begin{equation*}
-RWR^{T}+RR^{T}(RW^{-1}R^{T})^{-1}RR^{T}\preceq 0.
\end{equation*}
Applying the Schur complement formula, this inequality can be equivalently rewritten as
\begin{equation*}
\left[\begin{array}{cc}
-RWR^T & RR^T\\
RR^T & -RW^{-1}R^{T}
\end{array}\right]=\Theta\Psi\Theta^T\preceq 0,
\end{equation*}
where  
\begin{align*}
    \Theta=\left[\begin{array}{cc}
     R & 0\\
     0 & R
    \end{array}\right],~\Psi = 
    \left[\begin{array}{cc}
     -W & I\\
     I & -W^{-1}
    \end{array}\right].
\end{align*}
Therefore, our proof is concluded by showing that $\Psi\preceq 0$, which follows from the Schur complement formula.
\end{proof}

If $\lambda_{\rm{max}}(X_{1})$ is considered instead of $\lambda_{\rm{max}}(X)$, and given that $(L_{e,s}^{\tau})^{-1}\succ 0$, we apply the Schur complement formula to equivalently rewrite $X_{1}\preceq \zeta I$ as 
\begin{equation*}
\left[\begin{array}{cc}
\zeta I & W^{-\frac{1}{2}}R^{\dagger}\\
(W^{-\frac{1}{2}}R^{\dagger})^{T} & L_{e,s}^{\tau}
\end{array}\right]\succeq 0.
\end{equation*}
Therefore, we are able to replace the nonlinear constraint $X\preceq \zeta I$ by a linear matrix inequality (LMI) in $\zeta$, $W^{-\frac{1}{2}}$, and $E^{-1}$, where $L_{e,s}^{\tau}=D_{\tau}^{T}E^{-1}D_{\tau}$.

Moreover, to penalize small node time scales, two different methods can be used. The first method consists of adding the constraint $\sum_{i=1}^{n}\epsilon_{i}^{-1}\leq \mu$, where $\mu$ is a design parameter that ensures that the node time scales cannot be all equal to their minimum values. The second method consists of adding a regularization term $\|\epsilon^{-1}\|_{2}$ to the objective function instead. Two similar methods may be used to penalize large edge weights. The first method consists of adding the constraint $\sum_{l=1}^{|\mathcal{E}|}w_{l}^{-\frac{1}{2}}\geq \nu$, where $\nu$ is a design parameter that ensures that the edge weights cannot be all equal to their maximum values. The second method consists of adding $\|w\|_{2}$ to the objective function. However, since our decision variable is $w^{-\frac{1}{2}}$, a new variable $\xi\in\mathbb{R}^{|\mathcal{E}|}$ is introduced such that $w_{l}^{\frac{1}{2}}\leq \xi_{l}$, or equivalently, $W^{\frac{1}{2}}\preceq \Xi$, where $\Xi=\mathbf{diag}(\xi)$. Using the Schur complement formula, $W^{\frac{1}{2}}\preceq \Xi$ can be replaced by
\begin{equation*}
    \left[\begin{array}{cc}
     \Xi & I\\
     I & W^{{-}\frac{1}{2}}
    \end{array}\right]\succeq 0.
\end{equation*}
Thus, in the second method, the term $\|\xi\|_{2}$ is added to the objective function and the LMI above is added to the set of constraints.

Additionally, a heuristic based on the upper to lower bound ratio $\eta$ defined in (\ref{eq:ratio}) is proposed to tighten the bounds on the $\Hinf$-norm of $\Tilde{\Sigma}_{\tau}$. Noting that $\eta=\lambda_{\rm{max}}(W^{\frac{1}{2}})/\lambda_{\rm{min}}(W^{\frac{1}{2}})=\lambda_{\rm{max}}(W^{-\frac{1}{2}})/\lambda_{\rm{min}}(W^{-\frac{1}{2}})$ is a quasiconvex function of $W^{-\frac{1}{2}}\succ 0$, imposing an upper bound $\gamma$ on $\eta$ can be done through adding the convex constraint $\lambda_{\rm{max}}(W^{-\frac{1}{2}})-\gamma\lambda_{\rm{min}}(W^{-\frac{1}{2}})\leq 0$. 

The two different options for penalizing large edge weights and small time scales yield the formulation of four possible optimization problems. In this paper, we perform $\Hinf$-norm minimization by solving the following semidefinite program:
\renewcommand*{\arraystretch}{1.1}
\begin{align}
\tag{\mbox{$P_{1}$}}\label{opt} 
    \begin{array}{ll}
\underset{\zeta,w_{l}^{-\frac{1}{2}},\epsilon_{i}^{-1},\xi}{\text{minimize}} & \zeta+\alpha \|\xi\|_{2}+\beta  \|\epsilon^{-1}\|_{2}\\
    \text{subject~to}     & 
    \begin{bmatrix}
    \zeta I & W^{-1/2} R^\dag\\
    (R^\dag)^T W^{-1/2} & L_{e,s}^\tau
    \end{bmatrix} \succeq 0 \\
    &
    \begin{bmatrix}
    \Xi & I\\
    I & W^{-\frac{1}{2}}    
    \end{bmatrix} \succeq 0\\
    & w_{{\rm{max},}l}^{-\frac{1}{2}}\leq w_{l}^{-\frac{1}{2}}\leq w_{{\rm{min},}l}^{-\frac{1}{2}} \\
    & \epsilon_{{\rm{max},}i}^{-1}\leq\epsilon_{i}^{-1}\leq\epsilon_{{\rm{min},}i}^{-1}\\
    & 
    \lambda_{\rm{max}}(W^{-\frac{1}{2}})-\gamma\lambda_{\rm{min}}(W^{-\frac{1}{2}})\leq 0
    \end{array}
\end{align}
for all $i\in\mathcal{V}$ and $l\in\{1,\ldots,|\mathcal{E}|\}$, where $\alpha>0$ and $\beta>0$ are weights on the different components of the objective function.

\section{FORMATION CONTROL EXAMPLE}
\label{sec:example}

This section provides an example of the improvement of disturbance rejection observed by applying the edge weights and time scales obtained from Problem~\eqref{opt} with a specific choice of $(\alpha,\beta,\gamma)$. We start by performing a Pareto optimal front analysis to determine  a suitable choice of these parameters.

\subsection{Pareto Optimal Front of Problem~\eqref{opt}}

\begin{figure}
    \centering
    \includegraphics[scale=0.4]{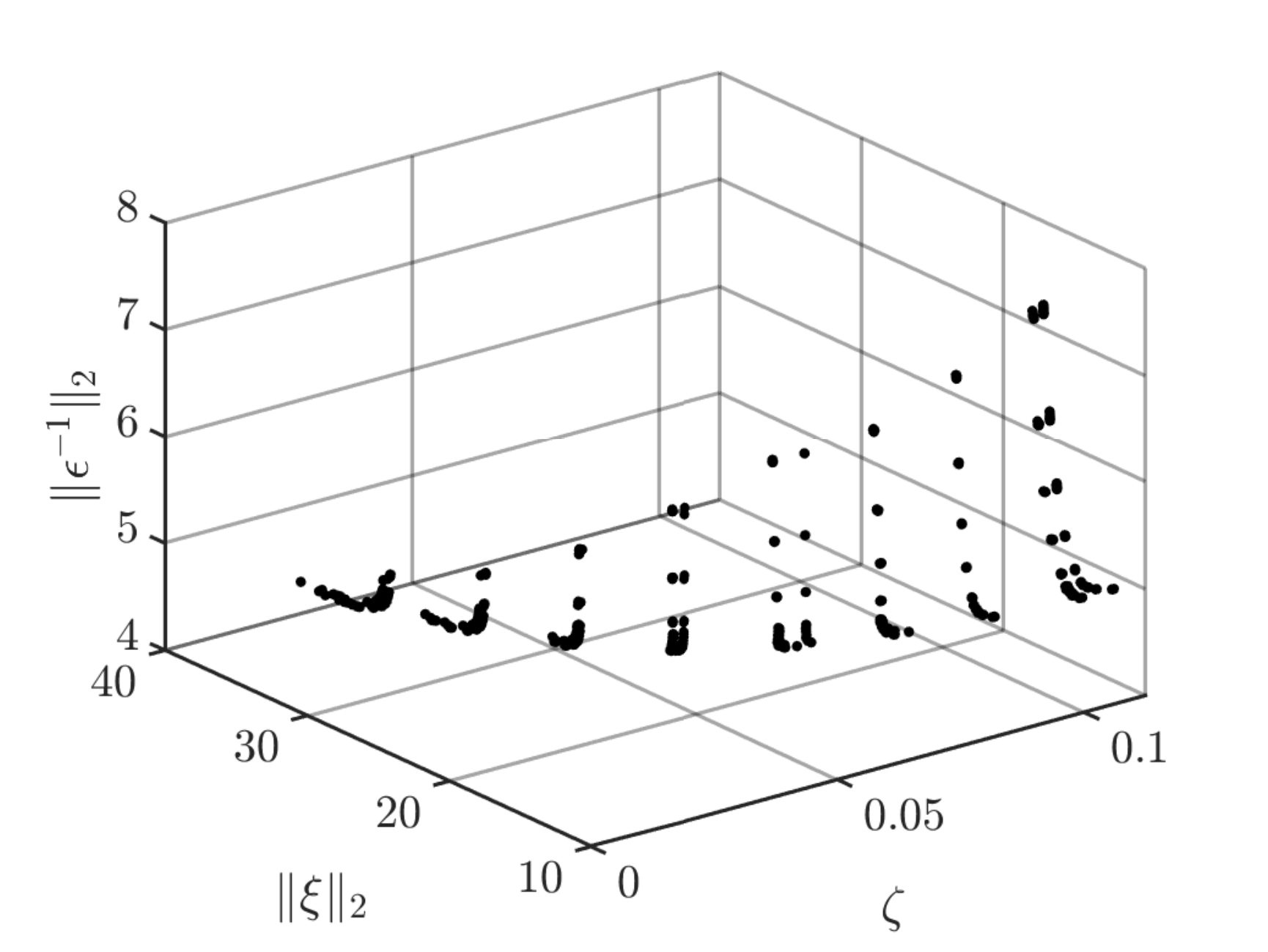}
    \includegraphics[scale=0.4]{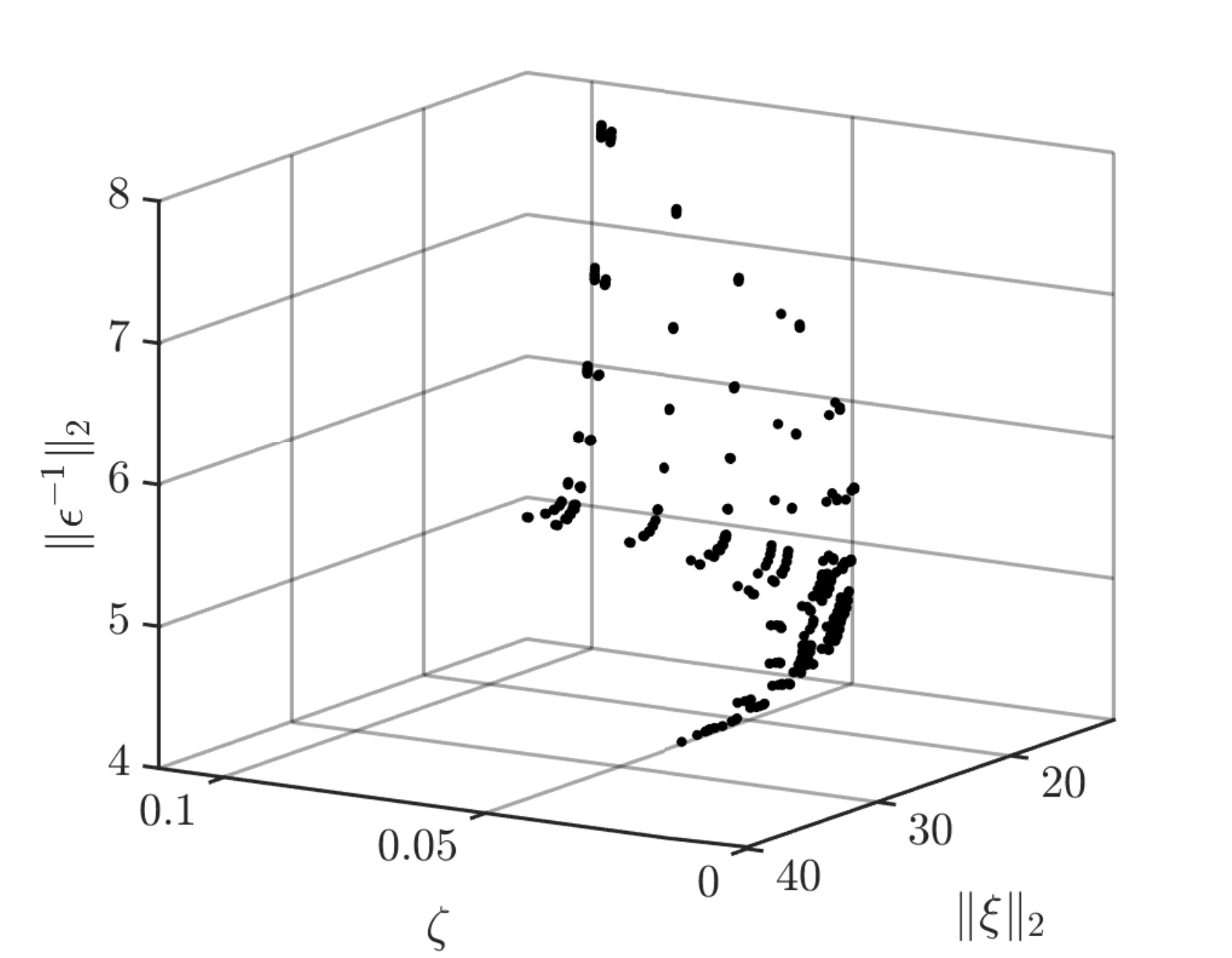}
    \includegraphics[scale=0.4]{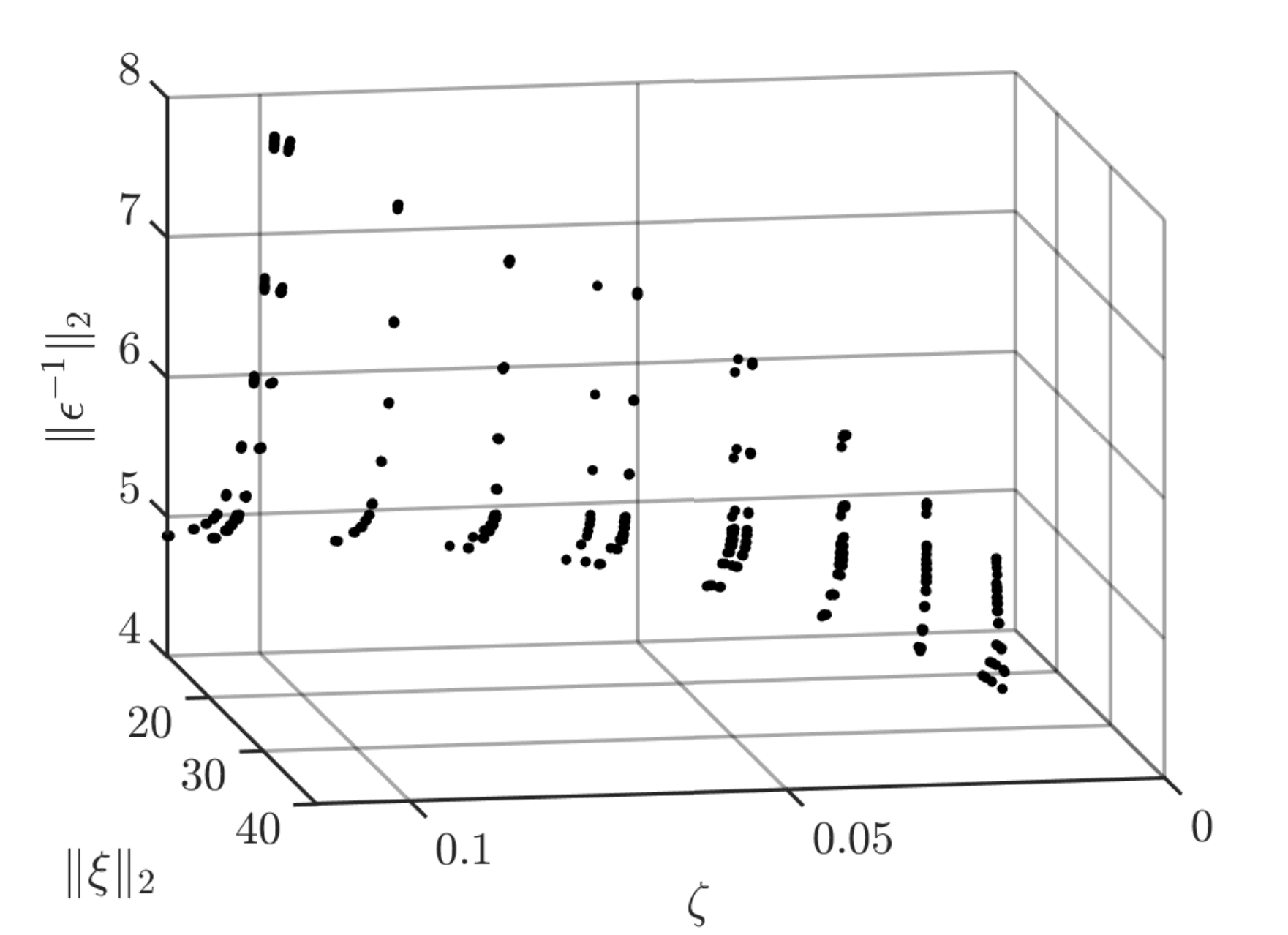}
    \caption{Pareto optimal front of Problem~\protect\eqref{opt} induced by the tuning parameters $(\alpha,\beta,\gamma)$.}
    \label{fig:pareto}
\end{figure}

A numerical experiment is performed on the graph in Figure~\ref{fig:graph} to examine the Pareto optimal front induced by the parameters $\alpha,$ $\beta,$ and $\gamma$.
The optimization problem~\eqref{opt} was solved over a logarithmic grid with $\alpha\in (10^{-3},10^{-1})$, $\beta \in (10^{-3},10^{-1})$, and $\gamma \in (10^{0.15},10^2)$.
Bounds on the edge weights are set as $w_{\min} = 10$ and $w_{\max} = 130$, and bounds on the time scales are set as either $\epsilon_{\min} = 0.1$ or $1$, and $\epsilon_{\max} = 0.5$ or $5$, depending on whether the node is a `fast' or `slow' node, as described in the example in \S\ref{sec:formcont}.

The resulting 3-dimensional surface corresponding to the Pareto optimal front and exhibiting the trade-off between the three competing objective functions $\zeta$, $\|\xi\|_{2}$, and $\|\epsilon^{-1}\|_{2}$ is shown in Figure~\ref{fig:pareto}.
A `knee' is observed at $(\alpha,\beta,\gamma) = (7.7\times 10^{-3}, 2.15\times 10^{-2},10)$.
In this setting, solving one instance of Problem~\eqref{opt} takes approximately 2.66 seconds on an Intel Core i7-9700K CPU (3.60GHz) using \texttt{YALMIP} and \texttt{SDPT3} \cite{YALMIP,SDPT3}.
A more specialized and/or optimized solver would lower this computational burden.

\subsection{Formation Control of Non-homogeneous Agents}\label{sec:formcont}

\begin{figure}
  \centering
  \includegraphics[width=\columnwidth]{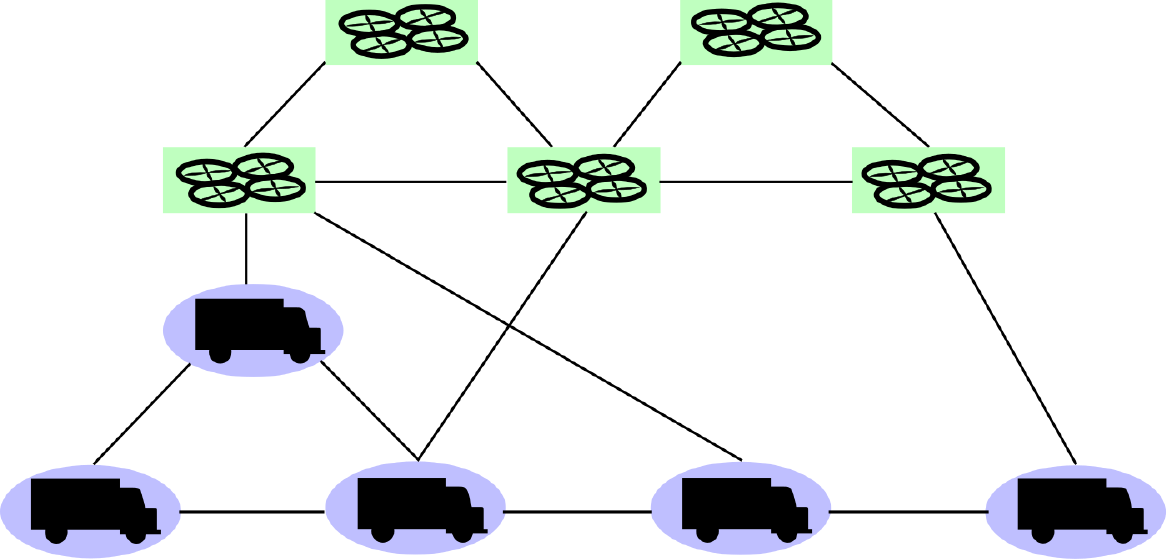}
  \caption{Network of non-homogeneous agents operating on multiple time scales. Faster aerial vehicles are shown in green squares, and slower ground vehicles are shown in blue ellipses.}
  \label{fig:formationControl}
\end{figure}

Consider a network of non-homogeneous agents, for example two groups consisting of autonomous ground and aerial vehicles, respectively, as depicted in Figure~\ref{fig:formationControl}. The underlying graph in Figure \ref{fig:formationControl} is the same as that of Figure \ref{fig:graph}, and the nodes and edges are labeled in the same way.
Denote the set of nodes corresponding to the ground vehicles as $\mathcal{V}_\text{g}$, and the set of nodes corresponding to the aerial vehicles as $\mathcal{V}_\text{a}$. 
Then, we can write $\mathcal{V}= \mathcal{V}_{\text{g}} \sqcup \mathcal{V}_{\text{a}}$.
These may represent autonomous delivery trucks coordinating with autonomous delivery drones, or aerial vehicles platooning around a ground vehicle formation.
A similar multi-time scale layered network configuration is considered for mobile sensor networks in~\cite{abuarqoub2017dynamic}.

\begin{figure*}
  \centering
  \includegraphics[width=\textwidth]{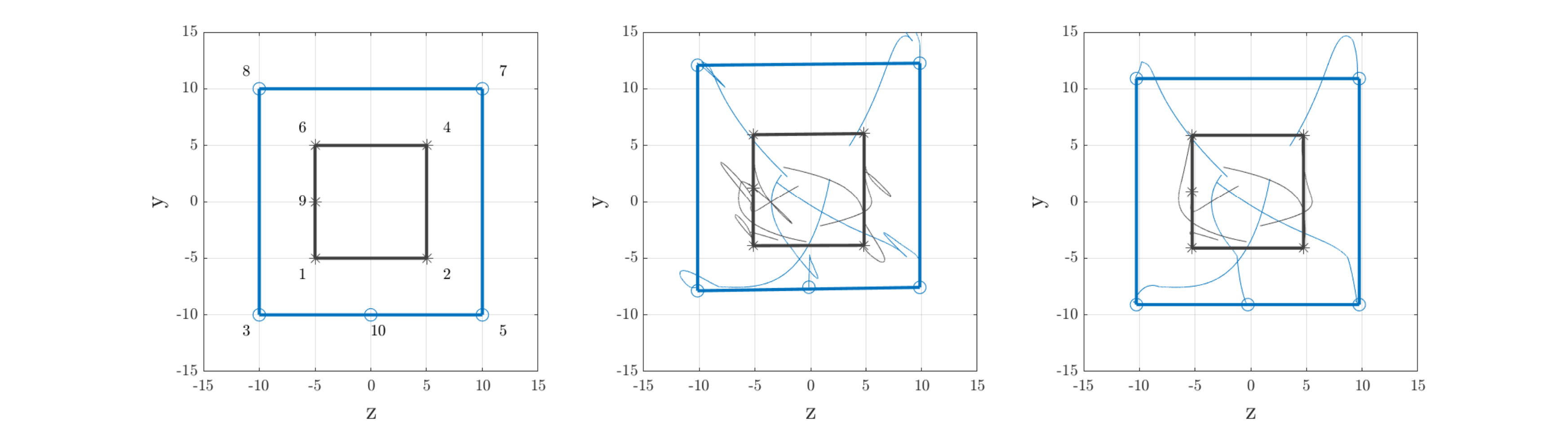}
  \caption{Target formation and agents trajectories. Left: Target formation, with node labels corresponding to Figure~\protect\ref{fig:graph}. Fast nodes are depicted on the outer square with `$\circ$' markers; slow nodes are depicted on the inner square with `$\ast$' markers. Middle: Trajectories of agents over 10s with no updates of time scales and edge weights. Right: Trajectories of agents over 10s with updates of both time scales and edge weights. Both simulations in the right and middle figures are subject to the same disturbances and initial conditions. Markers indicate the endpoints of the trajectories of the agents. Thick solid lines indicate the actual formation at the end of the 10s time frame.}
  \label{fig:formation}
\end{figure*}

Suppose that the natural dynamics of the vehicles operate on two ranges, i.e., let the time scales associated with the nodes corresponding to the (faster) aerial vehicles satisfy $\epsilon_{\min,\text{a}}  \leq \epsilon_i \leq \epsilon_{\max,\text{a}}$ for all $i\in \mathcal{V}_\text{a}$, and the time scales associated with the nodes corresponding to the (slower) ground vehicles satisfy $\epsilon_{\min,\text{g}}  \leq \epsilon_j \leq \epsilon_{\max,\text{g}}$ for all $j\in \mathcal{V}_{\text{g}}$.
Specific to our example, $\epsilon_{\min,\text{a}}=0.1$, $\epsilon_{\max,\text{a}}=0.5$, $\epsilon_{\min,\text{g}}=1$, and $\epsilon_{\max,\text{g}}=5$. 

\begin{figure}
  \centering
  \includegraphics[width=\columnwidth]{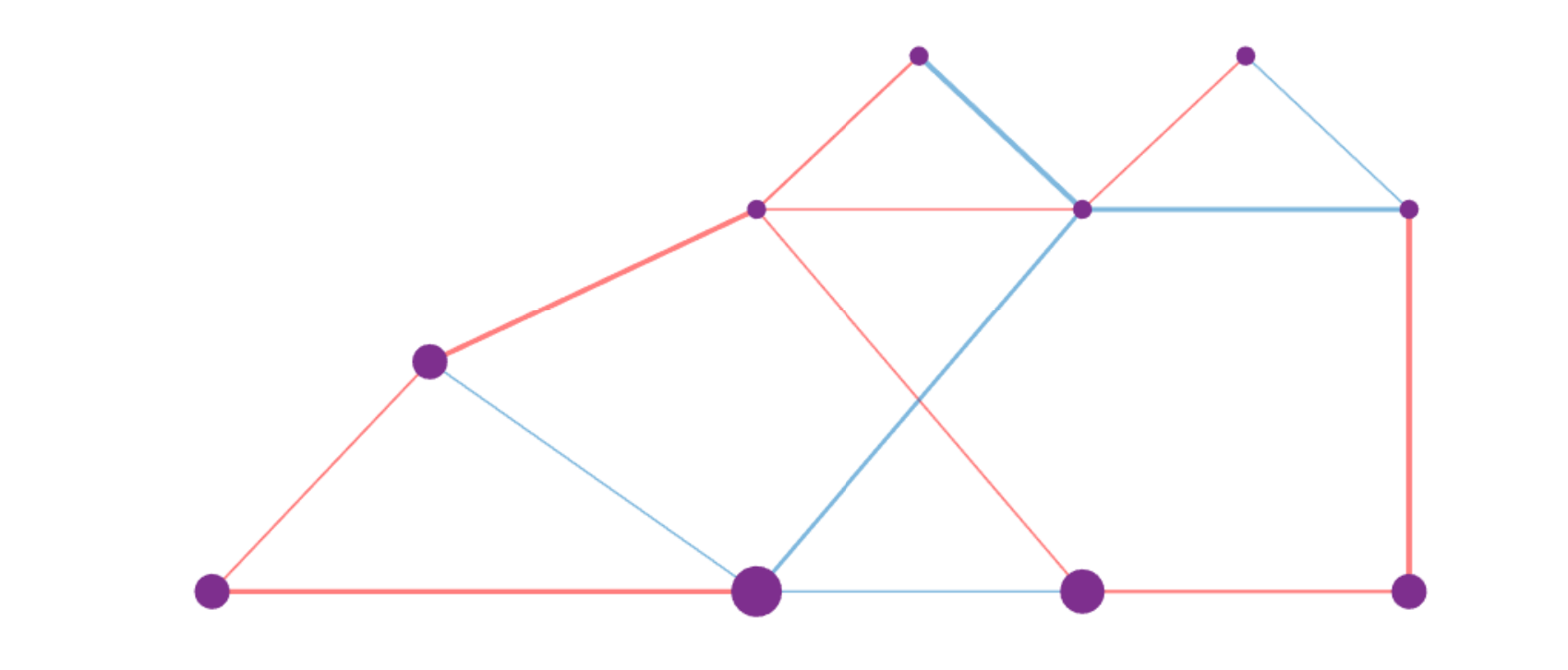}
  \caption{Optimal edge weights and time scales for the graph in Figure~\protect\ref{fig:graph} obtained from solving Problem~\protect\eqref{opt}  with $(\alpha,\beta,\gamma) = (7.7\times 10^{-3}, 2.15\times 10^{-2},10)$. Edges colored in red are the spanning tree edges. Edges colored in blue are the corresponding co-tree edges. Edge widths are proportional to the optimal weights on the edges, and node diameters are proportional to the optimal time scales on the nodes.}
  \label{fig:optTsW}
\end{figure}

Each agent $i$ in Figure~\ref{fig:formationControl} is assigned a position $(p_{z,i},p_{y,i})$ relative to a formation center in the plane, where the agents are numbered similarly as the corresponding nodes in Figure~\ref{fig:graph}. The formation is depicted in Figure~\ref{fig:formation}: slow agents are placed on the border of the inner square, and fast agents are placed on the border of the outer square.
In each direction $r\in\{z,y\}$, the agents run the one-dimensional consensus protocol:
\begin{align}
  \epsilon_i\dot{x}_{r,i} = \sum_{j\in N(i)}w_{ij} (x_j - x_i + p_{r,i} + d^e_{r,ij}) +d_{r,i}^n.
  \label{eq:consensus_example}
\end{align}
$d^{n}_{r,i}(t)$ and $d^{e}_{r,ij}(t)$ are disturbance signals on the nodes and edges, respectively, in the direction $r\in\{z,y\}$.
For all $i\in\mathcal{V}$, $d^{n}_{r,i}(t)$ is of the form
\begin{align}
  d_{r,i}^n(t) = 
  \begin{cases}
    a_{r,i} + b_{r,i} \cos\left(2\pi \frac{t-t_s}{t_f-t_s}\right) & \text{if}~t_s\leq t \leq t_f,\\
    0 & \text{otherwise},
  \end{cases}
    \label{eq:disturbance}
\end{align}
where $-3.5 \leq a_{r,i} \leq 3.5 $ and $-2.4 \leq b_{r,i}\leq 2.4$ are randomly chosen for each node and direction, and $[t_s,t_f] = [2,3]$ is the finite support of the disturbance. For all $l\in\{1,\ldots,|\mathcal{E}|\}$, $d^{e}_{r,ij}(t)$ is of the same form. The initial edge weights and time scales are set to unity, and Problem~\eqref{opt} is solved to improve the $\mathcal{H}_\infty$-norm.
We choose the parameters $(\alpha,\beta,\gamma) = (7.7\times 10^{-3}, 2.15\times 10^{-2},10)$ from the `knee' of the Pareto optimal front in Figure~\ref{fig:pareto}.
The resulting edge weights and time scales are depicted in Figure~\ref{fig:optTsW}.

Starting from randomly seeded initial positions, the consensus protocol~\eqref{eq:consensus_example} is simulated for $10$ seconds with the disturbances in~\eqref{eq:disturbance}. Two scenarios are considered. In one scenario, the time scales and edge weights are all set to unity throughout the simulation. In the second scenario, the time scales and edge weights are updated as per the solution of Problem~\protect\eqref{opt} during the time period $t\in[2,3]$ of the disturbance. One can see in Figure~\ref{fig:edgeStates} that the disturbance across the edge states is almost completely rejected when applying both the edge weight and time scale updates as described. 

The trajectories of the agents in the $(z,y)$ plane are depicted in Figure \ref{fig:formation} for both scenarios. For the scenario with the updates of both the edge weights and time scales, one can see the robustness of the reference trajectory tracking of the agents and the almost complete rejection of the disturbance.
Figure \ref{fig:formation} also shows the formation of the agents at the end of the simulations. 
The endpoints of the agents at the end of the simulation in the scenario with the updates of both the time scales and edge weights are closer to the desired formation than those corresponding to the scenario with non-updated non-optimized edge weights and time scales, indicating a more effective rejection of the disturbance.
A video of the positions of the agents over time in both scenarios is available at~\cite{oursoftware}.

\begin{figure}
  \centering
  \includegraphics[width=\columnwidth]{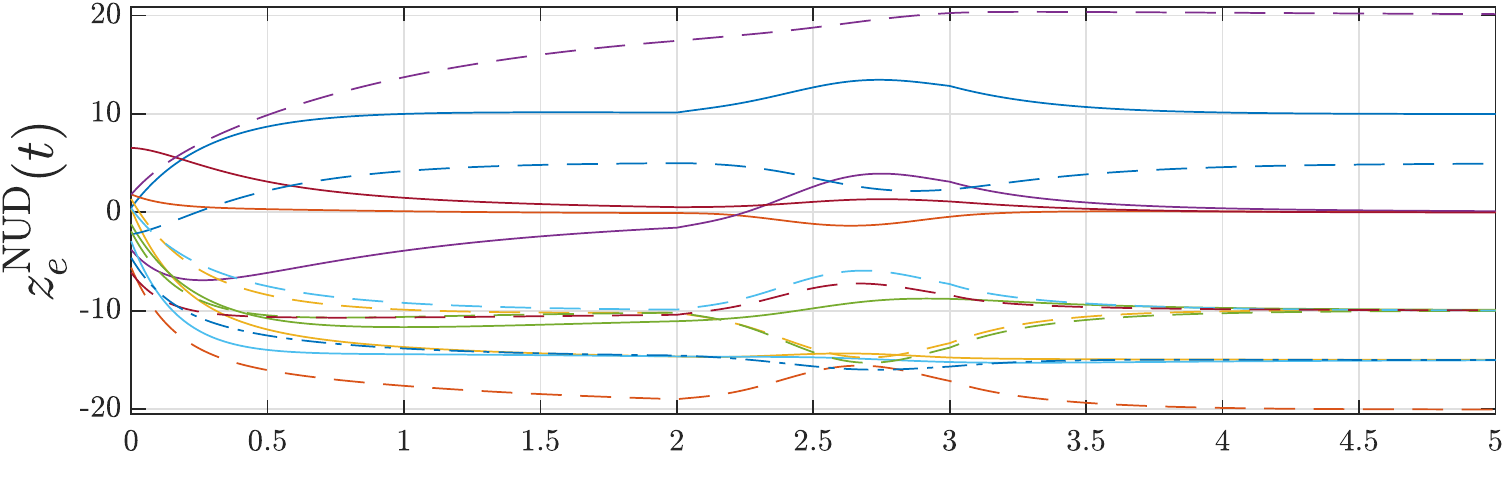}\\
  \includegraphics[width=\columnwidth]{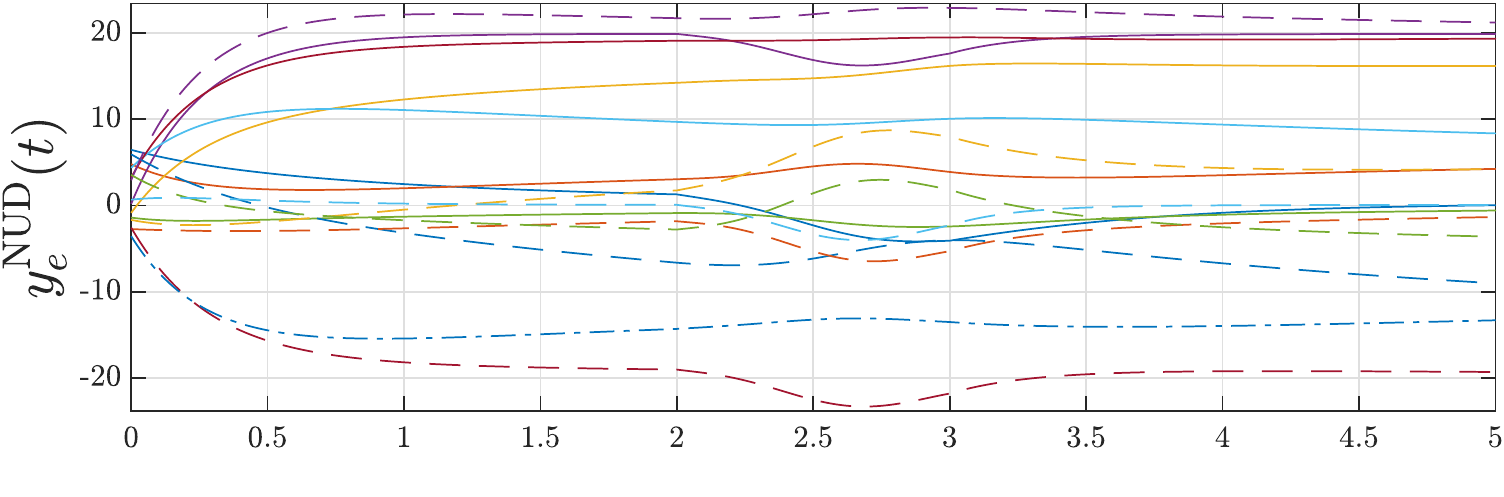}\\
  \includegraphics[width=\columnwidth]{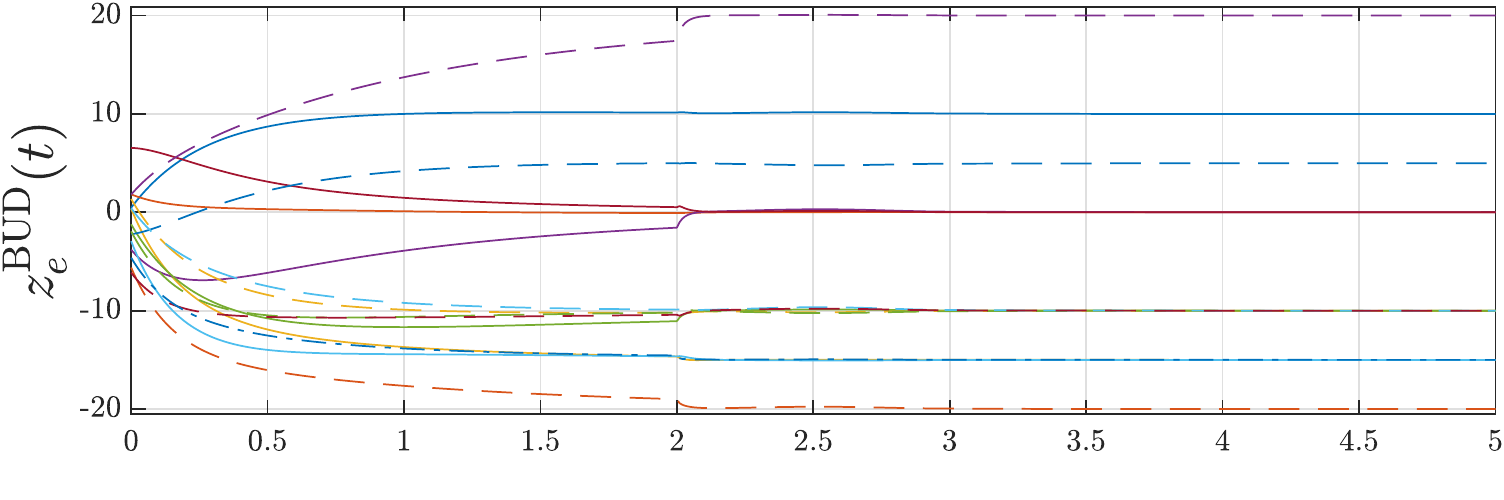}\\
  \includegraphics[width=\columnwidth]{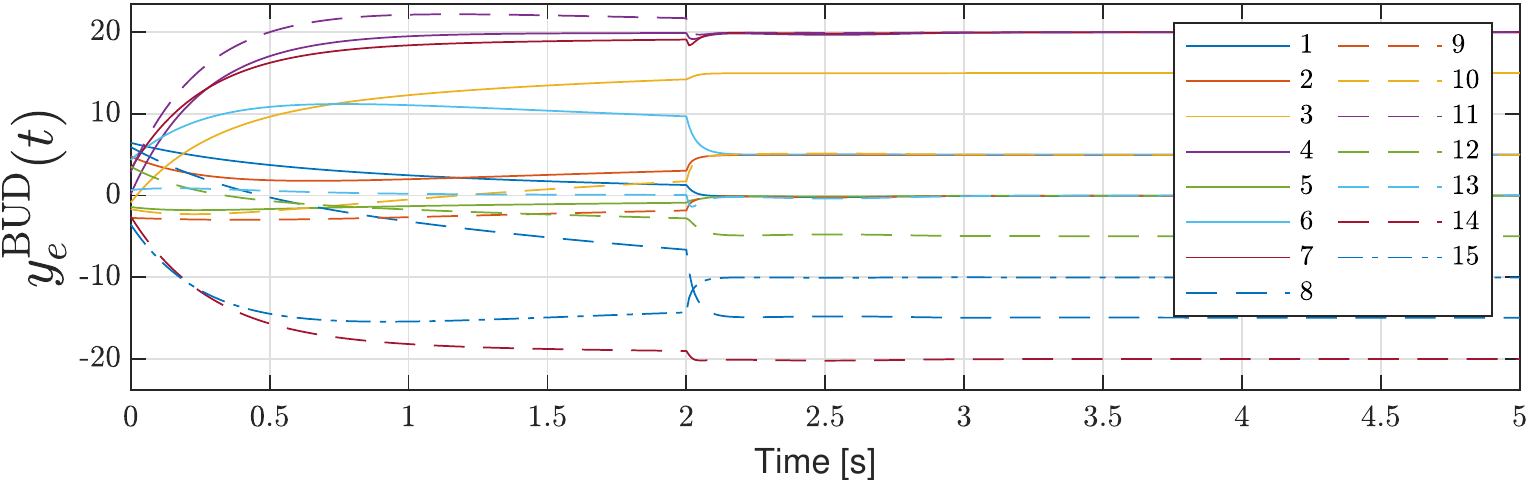}
  \caption{Edge states in $z,y$ directions over time. A disturbance is applied between $2\leq t\leq 3$. Test cases from top to bottom are with no updates (NUD) and both edge weight and time scale updates (BUD). Edge states are numbered in the legend in accordance to Figure \ref{fig:graph}.}
  \label{fig:edgeStates}
\end{figure}

\section{CONCLUSION}
\label{sec:conclusion}
This paper considers the $\Hinf$ performance problem for an edge variant of the consensus protocol on weighted graphs consisting of time-scaled nodes. 
Expressions of and bounds on the $\Hinf$-norm of the system of interest are provided. 
Specialized $\Hinf$-norm expressions are derived for the case of graphs with equal edge weights. 
Looser, but simpler, bounds on the $\Hinf$-norm are also derived for spanning tree graphs. 
Using the computed expressions and bounds, it is shown that the $\Hinf$-norm is minimized when the network is operated at the fastest time scales and largest edge weights. 
Furthermore, if such an operation is not possible/desirable (for example, in distributed systems consisting of non-homogeneous plants whose natural dynamics operate on different ranges of time scales), a versatile optimization setup is proposed for the selection of these parameters. 
The usefulness of the proposed optimization paradigm is illustrated via a formation control example with non-homogeneous agents. 
A potential direction for future research includes designing structure-preserving model reduction techniques, similar to the ones considered in \cite{AbouJaoudeBalancedTruncation,AbouJaoudeCoprimeFactors,AbouJaoudeLPV,AbouJaoudeCoprimeFactorsIJC}, that allow for reducing the system dynamics while retaining the network interpretation of the reduced-order system.

\end{document}